\documentclass[12pt]{amsart}

\usepackage[colorlinks=true, pdfstartview=FitV, linkcolor=blue, citecolor=blue, urlcolor=blue, breaklinks=true]{hyperref}
\usepackage{amsmath,amsfonts,amssymb,amsthm,amscd,comment,paralist,stmaryrd,etoolbox,mathtools,enumitem}
\usepackage[usenames,dvipsnames]{color}
\usepackage{pxfonts}

\makeatletter
\def\namedlabel#1#2{\begingroup
    #2%
    \def\@currentlabel{#2}%
    \phantomsection\label{#1}\endgroup
}

\newcommand{\pushright}[1]{\ifmeasuring@#1\else\omit\hfill$\displaystyle#1$\fi\ignorespaces}
\newcommand{\pushleft}[1]{\ifmeasuring@#1\else\omit$\displaystyle#1$\hfill\fi\ignorespaces}
\makeatother

\newcommand\Z{\mathbb{Z}}

\newcommand\N{\mathbb{N}}
\newcommand\F{\mathbb{F}}
\newcommand\kk{\Bbbk}

\newcommand\id{\mathrm{Id}}

\newcommand\fh{\mathfrak{h}}

\newcommand\cF{\mathcal{F}}
\newcommand\cG{\mathcal{G}}
\newcommand\cK{\mathcal{K}}

\newcommand\pj{\mathrm{proj}}

\newcommand\Vect{\mathrm{grVect}}

\newcommand\tr{\mathrm{tr}}

\newcommand{\md}{\textup{-mod}}
\newcommand{\pmd}{\textup{-pmod}}
\newcommand{\grm}{\textup{-grmod}}
\newcommand{\grp}{\textup{-grpmod}}

\newcommand\pR[1]{\prescript{R}{}{#1}}
\newcommand\pr[1]{\prescript{r}{}{#1}}
\newcommand\pL[1]{\prescript{L}{}{#1}}

\newcommand\qbin[2]{{#1 \atopwithdelims[] #2}}
\newcommand\Ner{N^{d,\epsilon}}

\newcommand\ts{\textstyle}

\DeclareMathOperator{\Hom}{Hom}
\DeclareMathOperator{\HOM}{HOM}
\DeclareMathOperator{\End}{End}

\DeclareMathOperator{\Res}{Res}
\DeclareMathOperator{\Ind}{Ind}

\DeclareMathOperator{\qdim}{grdim}
\DeclareMathOperator{\grdim}{grdim}

\newtheorem{theo}{Theorem}[section]
\newtheorem{prop}[theo]{Proposition}
\newtheorem{lem}[theo]{Lemma}
\newtheorem{cor}[theo]{Corollary}

\theoremstyle{definition}
\newtheorem{defin}[theo]{Definition}
\newtheorem{rem}[theo]{Remark}
\newtheorem{eg}[theo]{Example}

\numberwithin{equation}{section}
\allowdisplaybreaks

\newtoggle{comments}
\newtoggle{details}
\newtoggle{prelimnote}
\newtoggle{detailsnote}

\toggletrue{detailsnote}   % For including note about details toggle

\iftoggle{comments}{%
  \newcommand{\comments}[1]{
    \begin{center}
      \parbox{6.5 in}{
        \color{red}
          {\footnotesize \textbf{Comments:} #1}
        \color{black}}
    \end{center}}
}{%
  \newcommand{\comments}[1]{}
}

\iftoggle{details}{%
  \newcommand{\details}[1]{
      \ \\
      \color{OliveGreen}
        {\footnotesize \textbf{Details:} #1}
      \color{black}
      \\
  }
}{%
  \newcommand{\details}[1]{}
}

\iftoggle{prelimnote}{%
  \newcommand{\prelim}{\textsc{Preliminary version} \bigskip}
}{%
  \newcommand{\prelim}{}
}

\begin{document}
%
%%%%%%%%%%%%%%%%%%%%%%%%%%%%%%%%%%%

\title{Towers of graded superalgebras categorify the twisted Heisenberg double}

\author{Daniele Rosso}
\address{D.~Rosso: Department of Mathematics and Statistics, University of Ottawa, and Centre de Recherches Math\'ematiques, Montr\'eal}
\urladdr{\url{http://mysite.science.uottawa.ca/drosso/}}
\email{drosso@uottawa.ca}

\author{Alistair Savage}
\address{A.~Savage: Department of Mathematics and Statistics, University of Ottawa}
\urladdr{\url{http://AlistairSavage.ca}}
\email{alistair.savage@uottawa.ca}
\thanks{The second author was supported by a Discovery Grant from the Natural Sciences and Engineering Research Council of Canada.  The first author was supported by the Centre de Recherches Math\'ematiques and the Discovery Grant of the second author.}

\begin{abstract}
  We show that the Grothendieck groups of the categories of finitely-generated graded supermodules and finitely-generated projective graded supermodules over a tower of graded superalgebras satisfying certain natural conditions give rise to twisted Hopf algebras that are twisted dual.  Then, using induction and restriction functors coming from such towers, we obtain a categorification of the twisted Heisenberg double, introduced in~\cite{RS14}, and its Fock space representation.  We show that towers of wreath product algebras (in particular, the tower of Sergeev superalgebras) and the tower of nilCoxeter graded superalgebras satisfy our axioms.  In the latter case, one obtains a categorification of the quantum Weyl algebra.
\end{abstract}

%\date{\today}

\prelim

\maketitle
\thispagestyle{empty}

\tableofcontents

%%%%%%%%%%%%%%%%%%%%%%%%%%%%%%%%%%%%%%%%%%%%%%%%%%%%%%%%%%%%%%%%%%%%
%
\section{Introduction}
%
%%%%%%%%%%%%%%%%%%%%%%%%%%%%%%%%%%%%%%%%%%%%%%%%%%%%%%%%%%%%%%%%%%%%

The Heisenberg double is an associative algebra constructed in a natural way from a Hopf algebra.  As a $\kk$-module, the Heisenberg double $\fh(H^+,H^-)$ of a Hopf algebra $H^+$ over a ring $\kk$, with dual $H^-$, is isomorphic to $H^+ \otimes_\kk H^-$.  The factors $H^\pm$ are subalgebras, and there is an explicit commutation relation between the two factors.  The Heisenberg double can be viewed naturally as the subalgebra of $\End_\kk H^+$ generated by left multiplication by $H^+$ and the left regular action of $H^-$ on $H^+$.  From this point of view, one sees that the Heisenberg double has a natural representation on $H^+$, called the Fock space representation.  The most well-known example is when $H^+$ and $H^-$ are both the self-dual Hopf algebra of symmetric functions, in which case the Heisenberg double is the classical infinite-dimensional Heisenberg algebra.

In~\cite{RS14}, the authors introduced a twisted version of the Heisenberg double.  This involved replacing Hopf algebras by \emph{twisted Hopf algebras} (see Section~\ref{sec:general-construction}) and the Hopf pairing between $H^+$ and $H^-$ by a \emph{twisted Hopf pairing} (see Definition~\ref{def:hopf-pairing}).  If these data satisfy a certain natural compatibility condition (see Definition~\ref{def:compatible-pair}), one can construct the \emph{twisted Heisenberg double}, which also has a natural Fock space representation.  In the case that the twistings are trivial, one recovers the usual notion of the Heisenberg double.

One of the main motivations for the introduction of the twisted Heisenberg double in~\cite{RS14} was applications to categorification.  A conjectural categorification of the infinite-dimensional Heisenberg algebra itself was first described by Khovanov in~\cite{Kho10} (see also~\cite{LS13}).  Since then, the Heisenberg algebra has been categorified in several ways.  We refer the reader to the expository paper~\cite{LS12} for an overview of these approaches.  Motivated by these constructions and results on dual Hopf algebras obtained from towers of algebras (see~\cite{BL09}), a weak categorification of the more general Heisenberg double was developed in~\cite{SY13}.  The dual Hopf algebras $H^\pm$ correspond to the Grothendieck groups of the categories of finitely-generated modules and finitely-generated projective modules of a tower of algebras satisfying some natural conditions.  Then, using certain induction and restriction functors, one obtains a categorification of the associated Heisenberg double and its Fock space representation.

The goal of the current paper is to extend the results of~\cite{SY13} to the setting of towers of \emph{graded} \emph{super}algebras.  Since graded algebras and superalgebras are both ubiquitous in the categorification literature, the extension to this setting is a natural one.  Typically, passing to graded algebras results, on the level of Grothendieck groups, in a quantum deformation.  In the current paper, it results in the move from Hopf algebras to twisted Hopf algebras, and from the Heisenberg double to the twisted Heisenberg double.  On the other hand, generalizing from algebras to superalgebras introduces extra structure on the Grothendieck groups coming from the parity shift functor.  After background on the twisted Heisenberg double in Section~\ref{sec:general-construction} and on categories of supermodules over graded superalgebras in Section~\ref{sec:modules}, we introduce our main objects of study, towers of graded superalgebras, in Section~\ref{sec:towers}.  In Section~\ref{sec:categorification}, we present our main results, Theorem~\ref{theo:categorification} and Corollary~\ref{cor:Heis-double-cat}, that so-called \emph{strong} \emph{compatible} towers of graded superalgebras give rise to a (weak) categorification of the twisted Heisenberg double and its Fock space representation.

We also include in this paper several applications of our main result.  First, in Section~\ref{sec:frob-gr-alg}, we recall some facts about Frobenius graded superalgebras, using a slightly more general definition of these objects than is often considered in the literature.  We show that towers of Frobenius graded superalgebras automatically satisfy one of the important conditions of a strong tower.  Namely, induction is conjugate shifted right adjoint to restriction (see Proposition~\ref{prop:Frobenius-twisted}).  In Section~\ref{sec:wreath}, we discuss a large class of examples of strong compatible towers, namely, towers of wreath product algebras (see Corollary~\ref{cor:wreath-strong-compatible}).  This class includes the tower of Sergeev superalgebras and the algebras considered in the Heisenberg categorification of~\cite{CL12}.  We conclude, in Section~\ref{sec:Weyl}, with a discussion of the tower of nilCoxeter algebras, viewed as graded superalgebras.  This tower is also strong and compatible, and yields a categorification of the quantum Weyl algebra.

\iftoggle{detailsnote}{
\medskip

\paragraph{\textbf{Note on the arXiv version}} For the interested reader, the tex file of the arXiv version of this paper includes hidden details of some straightforward computations and arguments that are omitted in the pdf file.  These details can be displayed by switching the \texttt{details} toggle to true in the tex file and recompiling.
}{}

\medskip

%%%%%%%%%%%%%%%%%%%%%%
\subsection*{Notation}
%%%%%%%%%%%%%%%%%%%%%%

We let $\N$ and $\N_+$ denote the set of nonnegative and positive integers respectively. We let $\Z_2=\Z/2\Z$ be the ring of integers mod $2$.  We let $\F$ be a field of characteristic not equal to two.  By a slight abuse of terminology, we will use the terms \emph{module} and \emph{representation} interchangeably.

%%%%%%%%%%%%%%%%%%%%%%%%%%%%%%
\subsection*{Acknowledgements}
%%%%%%%%%%%%%%%%%%%%%%%%%%%%%%

The authors would like to thank A.~Licata, J.~Sussan, and O.~Yacobi for useful conversations.

%%%%%%%%%%%%%%%%%%%%%%%%%%%%%%%%%%%%%%%%%%%%%%%%%%%%%%%%%%%%%%%%%%%%
%
\section{The twisted Heisenberg double} \label{sec:general-construction}
%
%%%%%%%%%%%%%%%%%%%%%%%%%%%%%%%%%%%%%%%%%%%%%%%%%%%%%%%%%%%%%%%%%%%%

In this section, we review the concept, introduced in~\cite{RS14}, of the twisted Heisenberg double and its Fock space representation.   We fix a commutative ring $\kk$ and all algebras, coalgebras, bialgebras and Hopf algebras will be over $\kk$.  We will denote the multiplication, comultiplication, unit, counit and antipode of a Hopf algebra by $\nabla$, $\Delta$, $\eta$, $\varepsilon$ and $S$ respectively.  We write the product $\nabla$ as juxtaposition when this will not cause confusion, and use Sweedler notation
\[ \ts
  \Delta(a) = \sum_{(a)} a_{1} \otimes a_{2}
\]
for coproducts.  All tensor products are over $\kk$ unless otherwise indicated.

Let $(\Lambda,+)$ be a commutative monoid isomorphic (as a monoid) to $\N^r$.  Consider $H = \bigoplus_{\lambda \in \Lambda} H_\lambda$, where each $H_\lambda$, $\lambda \in \Lambda$, is finitely generated and free as a $\kk$-module, such that $(H,\nabla,\varepsilon)$ is a $\Lambda$-graded algebra and $(H,\Delta,\eta)$ is a $\Lambda$-graded coalgebra.  This means, in particular, that the following conditions are satisfied:
\begin{gather*} \ts
  \nabla(H_\lambda \otimes H_\mu) \subseteq H_{\lambda + \mu},\quad \Delta(H_\lambda) \subseteq \bigoplus_{\mu+\nu=\lambda} H_\mu \otimes H_\nu,\quad \lambda,\mu,\nu \in \Lambda, \\
  \eta(\kk) \subseteq H_0,\quad \varepsilon(H_\lambda) = 0 \text{ for } \lambda \in \Lambda\setminus\{0\}.
\end{gather*}

Fix $c \in \kk^\times$, and let $\chi=(\chi',\chi'')$ be a pair of biadditive maps $\chi',\chi'' \colon \Lambda \times \Lambda\to\Z$.  Define a new multiplication $*_\chi$ on $H \otimes H$ by the condition that, for homogeneous elements $a_i, b_i \in H$, $i=1,2$, we have
\begin{equation} \label{eq:twisted-mult}
  (a_1\otimes a_2) *_\chi (b_1\otimes b_2) = c^{\chi'(|a_2|,|b_1|)+\chi''(|a_1|,|b_2|)}a_1b_1\otimes a_2b_2,
\end{equation}
where $|a|$ denotes the degree of a homogeneous element $a \in H$.  (Whenever we write an expression involving $|a|$ for some $a \in H$, we implicitly assume that $a$ is homogeneous.)  Since $\chi',\chi''$ are biadditive, $*_\chi$ is associative, and we denote by $(H\otimes H)_\chi$ this \emph{twisted} associative algebra structure.

We say that $H$ is a \emph{$\Lambda$-graded connected twisted bialgebra}, or, more precisely, a \emph{$(c,\chi)$-bialgebra} if $H_0 = \kk 1_H$ and
\[
  \Delta \colon H \to (H\otimes H)_\chi
\]
is an algebra homomorphism.  A $\Lambda$-graded connected twisted bialgebra is always a \emph{twisted Hopf algebra}, that is, there exists a $\kk$-linear map $S \colon H\to H$ called the antipode that satisfies the condition
\[
  \nabla(\id\otimes S)\Delta=\eta\varepsilon=\nabla(S\otimes \id)\Delta.
\]
Thus we will also call such an object a \emph{$\Lambda$-graded connected $(c,\chi)$-Hopf algebra}, or simply a \emph{$(c,\chi)$-Hopf alegbra}.  When we wish to leave the data $(c,\chi)$ implied, we refer to it simply as a \emph{twisted Hopf algebra}.  We will write $(q,\chi',\chi'')$ instead of $(q,(\chi',\chi''))$ when we wish to make the components of $\chi$ explicit.

\begin{defin}[Twisted Hopf pairing]\label{def:hopf-pairing}
  Suppose $H$ and $H'$ are twisted Hopf algebras and $\gamma = (\gamma', \gamma'')$ is a pair of biadditive maps $\gamma', \gamma'' \colon \Lambda \times \Lambda \to \Z$ .  Then a \emph{$(c,\gamma)$-twisted Hopf pairing} of $H$ and $H'$ is a bilinear map $\langle -, - \rangle \colon H \times H' \to \kk$ such that $\langle -,-\rangle |_{H_\lambda\times H'_\mu}  \equiv 0$ when $\lambda, \mu \in \Lambda$, $\lambda\neq\mu$, and
  \begin{gather*} \ts
    \langle xy, a \rangle = c^{\gamma'(|x|,|y|)} \langle x \otimes y, \Delta(a) \rangle, \\ \ts
    \langle x, ab \rangle = c^{\gamma''(|a|,|b|)}\langle \Delta(x), a \otimes b \rangle, \\
    \langle 1_H, a \rangle = \varepsilon(a),\quad \langle x, 1_{H'} \rangle = \varepsilon(x),
  \end{gather*}
  for all homogeneous $x,y \in H$, $a,b \in H'$, where we define
  \[
    \langle -, - \rangle \colon (H \otimes H) \otimes (H' \otimes H') \to \kk,\quad \langle x \otimes y, a \otimes b \rangle = \langle x, a \rangle \langle y, b \rangle,\quad x,y \in H,\ a,b \in H'.
  \]
  We will write $(c,\gamma',\gamma'')$ instead of $(c,(\gamma',\gamma''))$ when we wish to make the components of $\gamma$ explicit.  A $(1,\gamma)$-twisted Hopf pairing (equivalently, a $(c,0,0)$-twisted Hopf pairing) is simply called a \emph{Hopf pairing}.
\end{defin}

\begin{defin}[Dual pair] \label{def:dual-pair}
  We say that $(H^+, H^-)$ is a \emph{$(c,\gamma)$-dual pair} of twisted Hopf algebras if $H^+$ and $H^-$ are both twisted Hopf algebras, and there exists a $(c,\gamma)$-twisted Hopf pairing $\langle -, - \rangle \colon H^- \times H^+ \to \kk$ such that $\langle -,- \rangle|_{H^{-}_\lambda \times H^+_\lambda}$, $\lambda \in \Lambda$, is a perfect pairing.  When we wish to leave the data $(c,\gamma)$ implicit, we say that the pair $(H^+,H^-)$ is \emph{twisted dual}.
\end{defin}

For a biadditive map $\zeta \colon \Lambda \times \Lambda \to \Z$, we define
\[
  \zeta^T \colon \Lambda \times \Lambda \to \Z,\quad \zeta^T(\lambda,\mu)=\zeta(\mu,\lambda),\quad \lambda,\mu \in \Lambda.
\]

\begin{lem}[{\cite[Lem.~2.8]{RS14}}] \label{lem:dual-twisted-Hopf}
  Suppose $(H^+,H^-)$ is a $(c,\gamma)$-dual pair of twisted Hopf algebras and $H^+$ is a $(c,\chi)$-Hopf algebra.  Then $H^-$ is a $(c,\xi$)-Hopf algebra, where
  \begin{equation} \label{eq:xi-dual}
    \xi = (\xi',\xi''),\quad \xi' = (\chi')^T + \gamma' - (\gamma'')^T,\quad \xi'' = \chi'' + \gamma' - \gamma''.
  \end{equation}
\end{lem}

For the remainder of this section, we fix a $(c,\gamma)$-dual pair $(H^+,H^-)$ of twisted Hopf algebras, where $H^{+}$ is a $(c,\chi)$-Hopf algebra and $H^-$ is a $(c,\xi)$-Hopf algebra, with $\xi$ given by~\eqref{eq:xi-dual}.

Any $a \in H^+$ defines an element $\pL{a} \in \End_\kk H^+$ by left multiplication.  Similarly, any $x \in H^-$ defines an element $\pR{x} \in \End_\kk H^- $ by right multiplication, whose adjoint $\pR{x}^*$ is an element of $\End_\kk H^+$.  (In the case that $H^+$ or $H^-$ is commutative, we often omit the superscript $L$ or $R$.)  In this way we have $\kk$-algebra homomorphisms
\begin{gather}
  H^+ \hookrightarrow \End_\kk H^+,\quad a \mapsto \pL{a}, \label{eq:H+action} \\
  H^- \hookrightarrow \End_\kk H^+,\quad x \mapsto \pR{x}^*. \label{eq:H-action}
\end{gather}
The action of $H^-$ on $H^+$ given by~\eqref{eq:H-action} is called the \emph{left regular action}.  The maps~\eqref{eq:H+action} and~\eqref{eq:H-action} are both injective.  For $x \in H^-$ and $a,b \in H^+$, we have, by \cite[Lem.~3.2]{RS14} and~\eqref{eq:xi-dual},
\begin{equation} \label{eq:H-general-relation} \ts
  \pR{x}^*(ab) = \sum_{(x)} c^{\gamma'(|a|-|x_1|,|x_2|) + \gamma'(|b|-|x_2|,|x_1|) + \gamma''(|x_1|,|x_2|) + \chi'(|x_1|, |b| - |x_2|) + \chi''(|a|-|x_1|,|x_2|)} \pR{x_1}^*(a) \pR{x_2}^*(b).
\end{equation}
\details{
  Using~\eqref{eq:xi-dual}, the exponent of $q$ in~\cite[Lem.~3.2]{RS14} (which corresponds to $c$ here) becomes
  \begin{align*}
    \gamma''&(|x_1|,|b|-|x_2|) + \gamma''(|a|,|x_2|) + \xi'(|b|-|x_2|,|x_1|)+\xi''(|a|-|x_1|,|x_2|) \\
    &= \gamma''(|x_1|,|b|-|x_2|) + \gamma''(|a|,|x_2|) + \chi'(|x_1|,|b|-|x_2|) + \gamma'(|b|-|x_2|,|x_1|) - \gamma''(|x_1|,|b|-|x_2|) \\
    &\pushright{+ \chi''(|a|-|x_1|,|x_2|) + \gamma'(|a|-|x_1|,|x_2|) - \gamma''(|a|-|x_1|,|x_2|)} \\
    &= \gamma''(|x_1|,|x_2|) + \chi'(|x_1|,|b|-|x_2|) + \gamma'(|b|-|x_2|,|x_1|) + \chi''(|a|-|x_1|,|x_2|) + \gamma'(|a|-|x_1|,|x_2|).
  \end{align*}
}

\begin{defin}[Compatible dual pair] \label{def:compatible-pair}
  We say that the $(c,\gamma)$-dual pair $(H^+,H^-)$ is \emph{compatible} if there is a choice of $\chi$ such that $H^+$ is a $(c,\chi)$-Hopf algebra and
  \begin{equation} \label{eq:compatibility}
    \chi' = - (\gamma')^T.
  \end{equation}
\end{defin}

If the pair $(H^+,H^-)$ is compatible and we choose $\chi$ to satisfy~\eqref{eq:compatibility}, then~\eqref{eq:H-general-relation} implies that, for all $x \in H^-$ and $a \in H^+$, we have,
\begin{equation} \ts
  \begin{split}
    \pR{x}^* \pL{a} &= \ts \sum_{(x)} c^{\gamma'(|a|-|x_1|, |x_2|) + \gamma''(|x_1|, |x_2|) + \chi''(|a|-|x_1|, |x_2|)} \pL{\left(\pR{x_1}^*(a)\right)} \pR{x_2}^* \\
    &= \ts \sum_{(x)} c^{\gamma''(|a|,|x_2|)+\xi''(|a|-|x_1|,|x_2|)} \pL{\left(\pR{x_1}^*(a)\right)} \pR{x_2}^*.
  \end{split}
\end{equation}

\begin{defin}[The twisted Heisenberg double] \label{def:h}
  Suppose the pair $(H^+,H^-)$ is compatible and we choose $\chi$ to satisfy~\eqref{eq:compatibility}.  We define the \emph{twisted Heisenberg double} $\fh(H^+,H^-)$ of $H^+$ as follows.  We set $\fh(H^+,H^-) = H^+ \otimes H^-$ as $\kk$-modules, and we write $a \# x$ for $a \otimes x$, $a \in H^+$, $x \in H^-$, viewed as an element of $\fh(H^+,H^-)$.  Multiplication is given by
  \begin{align}\label{eq:smash-product} \ts
    (a \# x)(b \# y) &:= \ts \sum_{(x)} c^{\gamma''(|b|,|x_2|) + \xi''(|b|-|x_1|,|x_2|)} a \pR{x_1}^*(b) \# x_2 y \\
    \notag &= \ts \sum_{(x),(b)} c^{\gamma''(|b|,|x_2|) + \xi''(|b|-|x_1|,|x_2|) + \gamma'(|b_1|,|b_2|)} \langle x_1, b_2 \rangle ab_1 \# x_2 y,
  \end{align}
  where the second equality follows from~\cite[Lem.~3.1]{RS14}.  We will often view $H^+$ and $H^-$ as subalgebras of $\fh(H^+,H^-)$ via the maps $a \mapsto a \# 1$ and $x \mapsto 1 \# x$ for $a \in H^+$ and $x \in H^-$.  Then we have $ax = a \# x$.  When the context is clear, we will simply write $\fh$ for $\fh(H^+,H^-)$.
\end{defin}

An element $v$ of an $\fh$-module $V$ is called a \emph{lowest weight} (resp.\ \emph{highest weight}) \emph{vacuum vector} if $\kk v \cong \kk$, as $\kk$-modules, and $H^-_\lambda v = 0$ (resp.\ $H^+_\lambda v = 0$) for all $\lambda \ne 0$.  The algebra $\fh$ has a natural (left) representation on $H^+$ given by
\[
  (a \# x)(b) = a \pR{x}^*(b),\quad a,b \in H^+,\ x \in H^-.
\]
We call this the \emph{lowest weight Fock space representation} of $\fh$ and denote it by $\cF = \cF(H^+,H^-)$.  Note that this representation is generated by the lowest weight vacuum vector $1 \in H^+$.  By~\cite[Th.~4.3(d)]{RS14}, the action of $\fh$ on $\cF$ is faithful.  Thus, we may view $\fh$ as the subalgebra of $\End_\kk H^+$ generated by $\pL{a}$, $a \in H^+$, and $\pR{x}^*$, $x \in H^-$.  If $X^+$ is a $\Lambda$-graded subalgebra of $H^+$ that is invariant under the left regular action of $H^-$ on $H^+$, then $X^+ \# H^-$ is a subalgebra of $\fh$ acting naturally on $X^+$.

%%%%%%%%%%%%%%%%%%%%%%%%%%%%%%%%%%%%%%%%%%%%%%%%%%%%
%
\section{Supermodules over graded superalgebras} \label{sec:modules}
%
%%%%%%%%%%%%%%%%%%%%%%%%%%%%%%%%%%%%%%%%%%%%%%%%%%%%

In this section, we collect some facts about categories of supermodules over graded superalgebras and the Grothendieck groups of these categories.  Recall that $\F$ is an arbitrary field of characteristic not equal to two.

%%%%%%%%%%%%%%%%%%%%%%%%%%%%%%%%%%%%%%%%%%%%%%%%%%%%%%%%%%%%%
\subsection{Supermodule categories and homomorphism spaces} \label{subsec:module-categories}
%%%%%%%%%%%%%%%%%%%%%%%%%%%%%%%%%%%%%%%%%%%%%%%%%%%%%%%%%%%%%

Let $B$ be a graded (more precisely, $\N$-graded) superalgebra over $\F$. By this we mean that
\[ \ts
  B \cong \bigoplus_{i\in\N,\, \epsilon\in\Z_2} B_{i,\epsilon} \quad \text{ with } \quad B_{i,\epsilon} B_{j,\tau} \subseteq B_{i+j,\epsilon+\tau}, \quad i,j\in\N,~\epsilon,\tau\in\Z_2.
\]
If $B$ is a graded superalgebra, a graded $B$-supermodule $M$ is analogously a $(\Z\times\Z_2)$-graded vector space over $\F$ such that
\[
  B_{i,\epsilon} M_{j,\tau} \subseteq M_{i+j,\epsilon+\tau} \quad \text{for all } i \in \N,\ j \in \Z,\ \epsilon,\tau \in \Z_2.
\]
(Throughout the paper, all (super)modules are left (super)modules unless otherwise specified.)  If $v$ is a homogeneous element in a $\Z$-graded (resp.\ $\Z_2$-graded) vector space, we will denote by $|v|$ (resp.\ $\bar v$) its degree.  Whenever we write an expression involving degrees of elements, we will implicitly assume that such elements are homogeneous.

For $M$, $N$ two $(\Z\times\Z_2)$-graded vector spaces over $\F$, we define a $(\Z\times\Z_2)$-grading on the space $\HOM_\F(M,N)$ of all $\F$-linear maps by setting $\HOM_\F(M,N)_{i,\epsilon}$ to be the subspace of all homogeneous maps of degree $(i,\epsilon)$. That is,
\[
  \HOM_\F(M,N)_{i,\epsilon} := \{\alpha \in \HOM_\F(M,N)\ |\  \alpha(M_{j,\tau}) \subseteq N_{i+j,\epsilon+\tau} \ \forall\ j \in \Z,\ \tau \in \Z_2\}.
\]
If $M$, $N$ are two graded $B$-supermodules, then we define the $\F$-vector space
\[
  \Hom_B(M,N) := \{f \in \HOM_\F(M,N)_{0,0}\ |\ f(bm) = b f(m) \ \forall\ b\in B,\ m\in M\}.
\]

For $B$ a finite-dimensional graded superalgebra, let $B\grm$ denote the category of finitely-generated graded left $B$-supermodules and let $B\grp$ denote the category of finitely-generated projective graded left $B$-supermodules.  In both categories, we take the morphisms from $M$ to $N$ to be $\Hom_B(M,N)$. For each $n\in\Z$, we have a degree shift functor
\[
  \{n\} \colon B\grm \to B\grm,\quad M\mapsto M\{n\}.
\]
Here $M\{n\}$ is the same underlying vector space as $M$, with the same $B$-action, but a new grading given by $M\{n\}_{i,\epsilon}=M_{i-n,\epsilon}$.  We have an analogous functor on the category of graded right $B$-supermodules, and hence on the category of graded $B$-superbimodules.  We also have a parity shift functor
\[
  \Pi \colon B\grm \to B\grm,\quad M \mapsto \Pi M,
\]
that switches the $\Z_2$-grading of the spaces, i.e.\ $(\Pi M)_{i,\epsilon}=M_{i,\epsilon+1}$. The action of $B$ on $\Pi M$ is given by $b\cdot m=(-1)^{\bar b}bm$, where $bm$ is the action on $M$.  The parity shift functor $\Pi$ for a right supermodule switches the $\Z_2$-grading as above but, unlike the case of left supermodules, it does not change the signs in the action.  For $n \in \Z$ and $\sigma \in \Z_2$, we define the functor
\[
  \{n,\sigma\} \colon B\grm \to B\grm,\quad M \mapsto M\{n,\sigma\} := \Pi^\sigma M\{n\}.
\]
Both the degree shift and parity shift functors leave morphisms unchanged.  Abusing notation, we will also sometimes use $\{n,\sigma\}$ to denote the map $M \to M\{n,\sigma\}$ that is the identity on elements of $M$.

For $M,N\in B\grm$, we also define the $(\Z\times\Z_2)$-graded $\F$-vector space
\begin{equation} \label{bighom} \ts
  \HOM_B(M,N) := \bigoplus_{n\in\Z, \epsilon \in \Z_2} \HOM_B(M,N)_{n, \epsilon},\quad \HOM_B(M,N)_{n, \epsilon} := \Hom_B(M, N\{-n,\epsilon\}).
\end{equation}
Note that we have an isomorphism of graded $\F$-vector spaces\pagebreak
\begin{align}
  \HOM_B(M\{i,\epsilon\},N\{j,\tau\}) &\cong \HOM_B(M,N)\{j-i,\tau+\epsilon\}, \label{eq:HOM-isom} \\
  f &\mapsto \{-i,\epsilon\} \circ f \circ \{i,\epsilon\}. \nonumber
\end{align}
and that, for $n \in \Z$, $\epsilon \in \Z_2$,
\[
  \HOM_B(M,N)_{n,\epsilon} = \{f \in \HOM_\F(M,N)_{n,\epsilon}\ |\ f(bm)=(-1)^{\epsilon \bar b} b f(m) \ \forall\ b\in B,\ m\in M\}.
\]

If $M$ is a graded right $B$-supermodule, we sometimes will use the notation $\pR{b}(m)=mb$ for the operator of right multiplication.  We will also need another family of operators involving a sign.  For each homogeneous $b \in B$, we define an $\F$-linear operator
\begin{equation} \label{eq:sign-right-action}
  \pr{b} \colon M\to M,\quad \pr{b}(m):=(-1)^{\bar b \bar m} mb \quad \text{for all homogeneous } m \in M.
\end{equation}
Notice that, if $\bar b=0$, then $\pr{b} = \pR{b}$.

If $M$ is a graded $(B_1,B_2)$-superbimodule for graded superalgebras $B_1$, $B_2$, and $N$ is a graded $B_1$-supermodule, then $\HOM_{B_1}(M,N)$ is a graded left $B_2$-supermodule via the action
\begin{equation} \label{eq:left-hom-action}
  b \cdot f = (-1)^{\bar b \bar f} f \circ \pr{b},\quad b \in B_2,\ f \in \HOM_{B_1}(M,N),
\end{equation}
and $\HOM_{B_1}(N,M)$ is a graded right $B_2$-supermodule via the action
\begin{equation} \label{eq:right-hom-action}
  f \cdot b = (-1)^{\bar b \bar f} (\pr{b}) \circ f,\quad b \in B_2,\ f \in \HOM_{B_1}(N,M).
\end{equation}
It is routine to verify that~\eqref{eq:HOM-isom} is an isomorphism of graded left $B_2$-supermodules and graded right $B_2$-supermodules under the actions~\eqref{eq:left-hom-action} and~\eqref{eq:right-hom-action}, respectively (when $M$ and $N$ have the appropriate structure).
\details{
  In the setup of~\eqref{eq:left-hom-action}, we have
  \[
    \{-i,\epsilon\} \circ (b \cdot f) \circ \{i,\epsilon\} = \{-i,\epsilon\} \circ ((-1)^{\bar b \bar f} f \circ \pr{b}) \circ \{i,\epsilon\} = (-1)^{\bar b \bar f} \{-i,\epsilon\} \circ f \circ \{i, \epsilon\} \circ \pr{b}.
  \]
  In the setup of~\eqref{eq:right-hom-action}, we have
  \[
    \{-i,\epsilon\} \circ (f \cdot b) \circ\{i,\epsilon\} = \{-i,\epsilon\} \circ ((-1)^{\bar b \bar f} (\pr{b}) \circ f) \circ \{i,\epsilon\} = (-1)^{\bar b \bar f} (\pr{b}) \circ \{-i,\epsilon\} \circ f \circ \{i,\epsilon\}.
  \]
}

If $M$ is a graded right $B_2$-supermodule, then it is also a graded $(\F,B_2)$-superbimodule and so the dual supermodule
\[
  M^\vee := \HOM_\F(M,\F)
\]
is a graded left $B_2$-supermodule via the action~\eqref{eq:left-hom-action}. If $M$ is a graded left $B_1$-supermodule, then $M^\vee$ is a graded right $B_1$-supermodule via the action
\begin{equation} \label{eq:left-module-dual-right-action}
  f \cdot b = f \circ \pL{b},\quad f \in M^\vee,\ b \in B_1.
\end{equation}
Thus, if $M$ is a graded $(B_1,B_2)$-superbimodule, then $M^\vee$ is a graded $(B_2,B_1)$-superbimodule with action given, for homogeneous $b_1 \in B_1$, $b_2 \in B_2$, $f \in M^\vee$, by
\begin{equation} \label{eq:BA-hom-bimod}
  b_2 \cdot f \cdot b_1 = (-1)^{\bar b_2 \bar f} f \circ \pr{b_2} \circ \pL{b_1}=(-1)^{\bar b_2 \bar f + \bar b_2 \bar b_1}f \circ \pL{b_1} \circ \pr{b_2} .
\end{equation}
Notice that, with this definition, $M^\vee$ is a $(B_2,B_1)$-bimodule since $((b_2 \cdot f) \cdot b_1)=(b_2 \cdot (f \cdot b_1))$.  The difference in signs according to the order only appears when we want to express the formula in terms of the operators $\pL{b_1}$ and $\pr{b_2}$, which do not commute.

Suppose $M$ is a graded left $B$-supermodule, $N$ is a graded right $B$-supermodule, and $\tau$ is a graded superalgebra automorphism of $B$.  Then we can define the twisted graded left $B$-supermodule $^\tau M$ and twisted graded right $B$-supermodule $N^\tau$, with actions given by
\begin{gather}
  \label{eq:left-twist} b \cdot m \mapsto \tau(b) m,\quad b \in B,\ m \in M, \\
  \label{eq:right-twist} n \cdot b \mapsto n \tau(b),\quad b \in B,\ n \in N,
\end{gather}
where juxtaposition denotes the original action of $B$ on $M$ and $N$.

If $A$ and $B$ are graded superalgebras, then their tensor product $A \otimes B$ is also a graded superalgebra, with product
\[
  (a_1 \otimes b_1)(a_2 \otimes b_2)=(-1)^{\bar b_1 \bar a_2} a_1a_2\otimes b_1b_2, \quad a_1,\ a_2\in A,\ b_1,\ b_2\in B.
\]
Notice that there is an isomorphism $A\otimes B\cong B\otimes A$ of graded superalgebras defined by
\[
  a\otimes b\mapsto (-1)^{\bar a \bar b}b\otimes a, \quad a\in A,\ b\in B.
\]
Given $M\in A\grm$ and $N\in B\grm$, we define the outer tensor product $M\boxtimes N \in (A \otimes B)\grm$ to be the vector space $M\otimes N$, with the action defined by
\[
  (a\otimes b)\cdot (m\otimes n)=(-1)^{\bar b \bar m} am \otimes bn, \quad a\in A,\ b\in B,\ m\in M,\ n\in N.
\]
The outer tensor product of morphisms is defined as follows.  For $M_1,M_2 \in A\grm$, $N_1,N_2 \in B\grm$, $f \in \HOM_A(M_1,M_2)$, and $g \in \HOM_B(N_1,N_2)$, we define
\begin{gather*}
  f \boxtimes g \in \HOM_{A \otimes B}(M_1 \boxtimes N_1, M_2 \boxtimes N_2), \\
  (f \boxtimes g)(m \otimes n) = (-1)^{\bar g \bar m} f(m) \otimes g(n),\quad m \in M_1,\ n \in N_1.
\end{gather*}
\details{
  We have
  \begin{align*}
    (f \boxtimes g)((a \otimes b) \cdot (m \otimes n)) &= (f \boxtimes g)\left( (-1)^{\bar b \bar m} (a \cdot m) \otimes (b \cdot n) \right) \\
    &= (-1)^{\bar b \bar m + \bar g \bar a + \bar g \bar m} f(a \cdot m) \otimes g (b \cdot n) \\
    &= (-1)^{\bar b \bar m + \bar g \bar a + \bar g \bar m + \bar f \bar a + \bar g \bar b} (a \cdot f(m)) \otimes (b \cdot g(n)) \\
    &= (-1)^{\bar f \bar b + \bar g \bar a + \bar g \bar m + \bar f \bar a + \bar g \bar b} (a \otimes b) \cdot (f(m) \otimes g(n)) \\
    &= (-1)^{(\bar f + \bar g)(\bar a + \bar b)} (a \otimes b) \cdot \left( (f \boxtimes g)(m \otimes n) \right)
  \end{align*}
  Thus, $f \boxtimes g \in \HOM_{A \otimes B}(M_1 \boxtimes N_1, M_2 \boxtimes N_2)$.
}
We have an isomorphism of $(\Z \otimes \Z_2)$-graded $\F$-vector spaces
\begin{equation} \label{eq:outer-tensor-hom-isom}
  \HOM_A(M_1,M_2) \otimes_\F \HOM_B(N_1,N_2) \cong \HOM_{A \otimes B}(M_1 \boxtimes N_1, M_2 \boxtimes N_2).
\end{equation}

%%%%%%%%%%%%%%%%%%%%%%%%%%%%%%%%
\subsection{Grothendieck groups} \label{subsec:grothendieck}
%%%%%%%%%%%%%%%%%%%%%%%%%%%%%%%%

Suppose that $B$ is a finite-dimensional graded superalgebra.  We define the Grothendieck groups
\[
  G'_0(B) = \cK_0(B\grm) \quad \text{and} \quad K'_0(B) = \cK_0(B\grp).
\]
Here $\cK_0(B\grm)$ is defined to be the quotient of the free $\Z$-module with generators corresponding to finitely-generated graded $B$-supermodules, by the $\Z$-submodule generated by $M_1-M_2+M_3$ for all short exact sequences $0\to M_1\to M_2\to M_3\to 0$  in $B\grm$.  In the same way, we define $\cK_0(B\grp)$, except we only consider projective objects (note that, in this case, we are in fact considering the split Grothendieck group, since all short exact sequences in $B\grp$ split).  If $\mathcal{C}$ is a category, we denote the class of an object $M \in \mathcal{C}$ in $\cK_0(\mathcal{C})$ by $[M]$.

Let
\[
  \Z_{q,\pi} := \Z[q,q^{-1},\pi]/(\pi^2-1).
\]
We define actions of $\Z_{q,\pi}$ on $G'_0(B)$ and $K'_0(B)$ by setting
\begin{gather*}
  \pi[M] = [\Pi M], \quad q^n[M] := [M\{n\}],\quad \text{for all } M \in B\grm,\ n \in \Z, \\
  \pi[M] = [\Pi M], \quad q^n[M] := [M\{-n\}],\quad \text{for all } M \in B\grp,\ n \in \Z.
\end{gather*}
By~\eqref{eq:HOM-isom}, there is a natural $\Z_{q,\pi}$-bilinear map
\begin{equation} \label{eq:KG'-innerprod}
  \langle -, - \rangle \colon K'_0(B) \otimes G'_0(B) \to \Z_{q,\pi},\quad \langle [P], [M] \rangle =\grdim_\F \HOM_B(P,M),
\end{equation}
where, for a $(\Z \times \Z_2)$-graded vector space $V=\bigoplus_{n\in\Z,\, \epsilon \in \Z_2} V_{n,\epsilon}$ over $\F$, we define the graded dimension
\[ \ts
  \grdim_\F V:=\sum_{n \in \Z,\, \epsilon \in \Z_2} q^n \pi^\epsilon \dim_\F V_{n,\epsilon} \in \Z_{q,\pi}.
\]
Note that, in some places in the literature, one defines $q^n[M] = [M\{n\}]$ for $M \in B\grp$, $n \in \Z$, so that the form \eqref{eq:KG'-innerprod} is $q$-sesquilinear.  We choose a different convention because it simplifies certain arguments and notation to follow.

Any simple graded $B$-supermodule is concentrated in one degree for the $\Z$-grading (recall that $B$ is $\N$-graded), and so, up to shift, it is isomorphic to an ungraded simple supermodule for the superalgebra $\tilde{B}_0:=B_{0,0}\oplus B_{0,1}$. Let $V_1,\dotsc,V_s$ be a complete list of nonisomorphic simple $\tilde{B}_0$-supermodules.  If $P_i$ is the projective cover, in the category $B\grm$, of $V_i$ for $i=1,\dotsc,s$, then $P_1,\dotsc,P_s$ is a complete list of nonisomorphic indecomposable projective graded $B$-supermodules, up to shift.
\details{First of all, every $V_i$ has a projective cover in $B\grm$. To obtain it, let $\bar{P}_i$ be the projective cover of $V_i$ in the category of finitely-generated supermodules for $\tilde{B}_0$ (see \cite[Prop.~12.2.12]{Kle05}), then let $P_i:=B\otimes_{\tilde{B}_0}\bar{P}_i$. It is not difficult to check that $P_i$ is the required projective cover.  Now, each $P\in B\grp$  decomposes as a finite direct sum of indecomposable projective supermodules (a direct summand of a projective is projective). An indecomposable projective supermodule $P$ has a unique simple quotient $V$, which sits in some degree, and $P$ is the projective cover of $V$.  Hence it has to be one of the $P_i$, up to degree shift.}

If $\F$ is algebraically closed, then, for all $1 \le i,j \le s$,
\begin{equation} \label{eq:pairing-simple-proj}
  \langle [P_i], [V_j] \rangle =
  \begin{cases}
    \delta_{ij} & \text{ if }V_j \text{ is of type } \mathsf{M}, \\
    (1 + \pi)\delta_{ij} & \text{ if }V_j\text{ is of type } \mathsf{Q}.
  \end{cases}
\end{equation}
(For this result and the definition of types $\mathsf{M}$ and $\mathsf{Q}$, see \cite[Lem.~12.2.3]{Kle05}.)  If $V_1,\ldots, V_r$ are irreducibles of type $\mathsf{M}$ and $V_{r+1},\ldots, V_s$ are irreducibles of type $\mathsf{Q}$, we have isomorphisms of $\Z_{q,\pi}$-modules
\begin{gather*} \ts
  G'_0(B) \cong \bigoplus_{i=1}^r \Z_{q,\pi}[V_i]\oplus \bigoplus_{j=r+1}^s \left(\Z_{q,\pi}/(\pi-1)\right) [V_j], \\ \ts
  K'_0(B) \cong \bigoplus_{i=1}^r \Z_{q,\pi}[P_i]\oplus \bigoplus_{j=r+1}^s \left(\Z_{q,\pi}/(\pi-1)\right) [P_j].
\end{gather*}

Note that, if $B$ has graded left supermodules of type $\mathsf{Q}$, then the bilinear form~\eqref{eq:KG'-innerprod} is degenerate, since $1+\pi$ is a zero divisor.  Furthermore, $K'_0(B)$ and $G'_0(B)$ are not free $\Z_{q,\pi}$-modules.  We remedy this situation as follows.  Let
\[
  \kk =
  \begin{cases}
    \Z_{q,\pi} & \text{if all simple graded left $B$-supermodules are of type $\mathsf{M}$}, \\
    \Z[\frac{1}{2},q,q^{-1}] & \text{if $B$ has a simple graded left supermodule of type $\mathsf{Q}$}.
  \end{cases}
\]
We will identify $\Z[\frac{1}{2},q,q^{-1}]$ with $\Z_{q,\pi}[\frac{1}{2}]/(\pi-1)$.  Thus, $\Z[\frac{1}{2},q,q^{-1}]$ is naturally a $\Z_{q,\pi}$-module and we have $\pi=1$ in $\Z[\frac{1}{2},q,q^{-1}]$.  Define
\[
  G_0(B) = G'_0(B) \otimes_{\Z_{q,\pi}} \kk,\quad K_0(B) = K'_0(B) \otimes_{\Z_{q,\pi}} \kk.
\]
Then~\eqref{eq:KG'-innerprod} induces a pairing
\begin{equation} \label{eq:KG-innerprod}
  \langle -, - \rangle \colon K_0(B) \otimes G_0(B) \to \kk.
\end{equation}
If $\F$ is algebraically closed, we have
\[ \ts
  G_0(B) \cong \bigoplus_{i=1}^s \kk[V_i], \quad K_0(B) \cong \bigoplus_{i=1}^s \kk[P_i],
\]
and the pairing~\eqref{eq:KG-innerprod} is perfect.

Suppose $\varphi \colon B \to A$ is a graded superalgebra homomorphism, possibly not preserving the identity element. In particular $\varphi \in \HOM_\F(B,A)_{0,0}$ because $\varphi(1_B)\in A_{0,0}$ is an idempotent.  Then we can consider $A$ as a graded left $B$-supermodule via the action $b \cdot a = \varphi(b) a$.  Similarly, we can consider $A$ as a right $B$-supermodule.  Then we have \emph{induction} and \emph{restriction} functors
\begin{gather*}
  \Ind^A_B \colon B\grm \to A\grm, \qquad  \Res^A_B \colon A\grm \to B\grm,\\
  \Ind^A_B N := A\varphi(1_B) \otimes_B N,\quad N \in B\grm, \\
 \Res^A_B M := \HOM_A(A\varphi(1_B),M) \cong {_B\varphi(1_B)A} \otimes_A M,\quad M \in A\grm,
\end{gather*}
where $_B\varphi(1_B)A$ denotes $\varphi(1_B)A$ considered as a $(B,A)$-superbimodule and the left $B$-action on $\HOM_A(A\varphi(1_B),M)$ is given by using the right action of $B$ on $A$, as in \eqref{eq:left-hom-action}.  The isomorphism above is given by the map $f \mapsto \varphi(1_B) \otimes f(\varphi(1_B))$ for $f \in \HOM_A(A\varphi(1_B),M)$.  This isomorphism is natural in $M$, and so we have an isomorphism of functors $\Res^A_B \cong {_B\varphi(1_B)A} \otimes_A -$.  It follows from the standard tensor-hom adjunction that $\Ind_B^A$ is left adjoint to $\Res_B^A$.

%%%%%%%%%%%%%%%%%%%%%%%%%%%%%%%%%%%%%%%%%%%%%%%%%%%%
%
\section{Towers of graded superalgebras} \label{sec:towers}
%
%%%%%%%%%%%%%%%%%%%%%%%%%%%%%%%%%%%%%%%%%%%%%%%%%%%%

In this section, we define towers of graded superalgebras and various properties they may possess.  These towers will be the main ingredient in our categorification.  Recall that $\F$ is an arbitrary field of characteristic not equal to two.

\begin{defin}[Tower of graded superalgebras] \label{def:tower}
  Let $\Lambda\cong \N^r$ be as in the beginning of Section \ref{sec:general-construction}, and let $A = \bigoplus_{\lambda \in \Lambda} A_\lambda$ be a $\Lambda$-graded superalgebra over $\F$ with multiplication $\rho \colon A \otimes A \to A$.  Define
  \[
    \kk =
      \begin{cases}
        \Z_{q,\pi} & \text{if, for all $\lambda \in \Lambda$, all simple graded left $A_\lambda$-supermodules are of type $\mathsf{M}$}, \\
        \Z[\frac{1}{2},q,q^{-1}] & \text{if, for some $\lambda \in \Lambda$, $A_\lambda$ has a simple graded left supermodule of type $\mathsf{Q}$}.
      \end{cases}
  \]
  Then $A$ is called a \emph{tower of graded superalgebras} if the following conditions are satisfied.
  \begin{description}[style=multiline]
    \item[\namedlabel{item:TA1}{TA1}] Each graded piece $A_\lambda$, $\lambda \in \Lambda$, is a finite-dimensional $\N$-graded superalgebra as in Section~\ref{subsec:module-categories} (with a different multiplication than that of $A$) with a unit $1_\lambda$.  We have $A_0 \cong \F$.
    \item[\namedlabel{item:TA2}{TA2}] The external multiplication $\rho_{m,n} \colon A_\lambda \otimes A_\mu \to A_{\lambda+\mu}$ is a homomorphism of graded superalgebras for all $\lambda,\mu \in \Lambda$ (possibly not sending $1_\lambda \otimes 1_\mu$ to $1_{\lambda+\mu}$).
    \item[\namedlabel{item:TA3}{TA3}] We have that $A_{\lambda+\mu}$ is a two-sided projective $(A_\lambda \otimes A_\mu)$-supermodule with the action defined by
        \[
          a \cdot (b \otimes c) = a \rho_{\lambda,\mu}(b \otimes c) \quad \text{and} \quad (b \otimes c) \cdot a = \rho_{\lambda,\mu}(b \otimes c) a,
        \]
        for all $\lambda,\mu \in \Lambda$, $a \in A_{\lambda+\mu}$, $b \in A_\lambda$, $c \in A_\mu$.
    \item[\namedlabel{item:TA4}{TA4}] For each $\lambda \in \Lambda$, the pairing~\eqref{eq:KG-innerprod}, with $B=A_\lambda$, is perfect. (Note that this condition is automatically satisfied if $\F$ is algebraically closed.)
  \end{description}
\end{defin}

For the remainder of this section we assume that $A$ is a tower of graded superalgebras.  Let
\begin{equation} \label{eq:G0A-def} \ts
  \cG(A) = \bigoplus_{\lambda \in \Lambda} G_0(A_\lambda) \quad \text{and} \quad \cK(A) = \bigoplus_{\lambda \in \Lambda} K_0(A_\lambda).
\end{equation}
Then we have a perfect pairing $\langle -, - \rangle \colon \cK(A) \times \cG(A) \to \kk$ given by
\begin{equation} \label{eq:tower-pairing}
  \langle [P], [M] \rangle =
  \begin{cases}
    \qdim_\F \HOM_{A_{\lambda}}(P,M) & \text{if } P \in A_\lambda\grp \text{ and } M \in A_\lambda\grm \\ & \quad \text{for some } \lambda \in \Lambda, \\
      0 & \text{otherwise}.
  \end{cases}
\end{equation}
We also define a perfect pairing $\langle -, - \rangle \colon (\cK(A) \otimes \cK(A)) \times (\cG(A) \otimes \cG(A)) \to \kk$ by
\[
  \langle [P] \boxtimes [Q], [M] \boxtimes [N] \rangle =
  \begin{cases}
    \qdim_\F (\HOM_{A_\lambda \otimes A_\mu}(P \boxtimes Q, M \boxtimes N)) \quad  \text{ if } P \in A_\lambda\grp, Q \in A_\mu\grp \\ \qquad \text{ and } M \in A_\lambda\grm,  N \in A_\mu\grm  \text{ for some } \lambda,\mu \in \Lambda, \\
    0 \qquad \text{ otherwise}.
  \end{cases}
\]
It follows from~\eqref{eq:outer-tensor-hom-isom} that, for all $P,Q,M,N$,
\begin{equation} \ts
  \langle [P] \boxtimes [Q], [M] \boxtimes [N] \rangle = \langle [P], [M] \rangle \langle [Q], [N] \rangle.
\end{equation}

Consider the direct sums of categories
\[ \ts
  A\grm := \bigoplus_{\lambda \in \Lambda} A_\lambda\grm,\quad A\grp := \bigoplus_{\lambda \in \Lambda} A_\lambda\grp.
\]
For $r \in \N_+$, we define
\begin{gather*} \ts
  A\grm^{\otimes r} := \bigoplus_{\lambda_1,\dotsc,\lambda_r \in \Lambda} (A_{\lambda_1} \otimes \dotsb \otimes A_{\lambda_r})\grm,\\ \ts
  A\grp^{\otimes r} := \bigoplus_{\lambda_1,\dotsc,\lambda_r \in \Lambda} (A_{\lambda_1} \otimes \dotsb \otimes A_{\lambda_r})\grp.
\end{gather*}
Then, for $i,j \in \{1,\dotsc,r\}$, $i < j$, we define $S_{ij} \colon A\grm^{\otimes r} \to A\grm^{\otimes r}$ to be the endofunctor that interchanges the $i$th and $j$th factors, that is, the endofunctor arising from the isomorphism
\[
  A_{\lambda_1} \otimes \dotsb \otimes A_{\lambda_r} \cong A_{\lambda_1} \otimes \dotsb \otimes A_{\lambda_{i-1}} \otimes A_{\lambda_j} \otimes A_{\lambda_{i+1}} \otimes \dotsb \otimes A_{\lambda_{j-1}} \otimes A_{\lambda_i} \otimes A_{\lambda_{j+1}} \otimes \dotsb \otimes  A_{\lambda_r}
\]
that maps the element $a_1 \otimes \dotsb \otimes a_r$ to
\[
   (-1)^{(\bar a_i + \bar a_j)(\bar a_{i+1} + \dotsb + \bar a_{j-1}) + \bar a_i \bar a_j} a_1 \otimes \dotsb \otimes a_{i-1} \otimes a_j \otimes a_{i+1} \otimes \dotsb \otimes a_{j-1} \otimes a_i \otimes a_{j+1} \otimes \dotsb \otimes a_{r}.
\]
We use the same notation to denote the analogous endofunctor on $A\grp^{\otimes r}$.

We also have the following functors:
\begin{equation} \label{eq:Groth-operations}
  \begin{split}
  \nabla \colon A\grm^{\otimes 2} \to A\grm,\quad \nabla|_{(A_\lambda \otimes A_\mu)\grm} = \Ind^{A_{\lambda + \mu}}_{A_\lambda \otimes A_\mu}, \\ \ts
  \Delta \colon A\grm \to A\grm^{\otimes 2}, \quad \Delta|_{A_\lambda\grm} = \bigoplus_{\mu+\nu =\lambda } \Res^{A_\lambda}_{A_\mu \otimes A_\nu}, \\
  \eta \colon \Vect \to A\grm,\quad \eta(V) = V \in A_0\grm \text{ for } V \in \Vect, \\
  \varepsilon \colon A\grm \to \Vect,\quad \varepsilon(V) =
  \begin{cases}
    V & \text{if } V \in A_0\grm, \\
    0 & \text{otherwise}.
  \end{cases}
  \end{split}
\end{equation}
In the above, we have identified $A_0\grm$ with the category $\Vect$ of finite-dimensional $(\Z \times \Z_2)$-graded vector spaces over $\F$.  Replacing $A\grm$ by $A\grp$ above, we also have the functors $\nabla$, $\Delta$, $\eta$ and $\varepsilon$ on $A\grp$. Since the above functors are all exact (we use axiom~\ref{item:TA3} here), they induce a multiplication, comultiplication, unit and counit on $\cG(A)$ and $\cK(A)$.  We use the same notation to denote these induced maps.

Since induction is left adjoint to restriction, $\nabla$ is left adjoint to $\Delta$.  However, there are examples of towers of graded superalgebras for which induction is not \emph{right} adjoint to restriction (see, for example, Section~\ref{sec:Weyl}).  Nevertheless, we often have something quite close to this property.  Any graded superalgebra automorphism $\psi_\lambda$ of $A_\lambda$ induces an isomorphism of categories $\Psi_\lambda \colon A_\lambda\grm \to A_\lambda\grm$ (which restricts to an isomorphism of categories $\Psi_\lambda \colon A_\lambda\grp \to A_\lambda\grp$) by twisting the $A_\lambda$-action as in~\eqref{eq:left-twist}.  Then $\Psi := \bigoplus_{\lambda \in \Lambda} \Psi_\lambda$ is an automorphism of the categories $A\grm$ and $A\grp$.  It induces automorphisms, which we also denote by $\Psi$, of $\cG(A)$ and $\cK(A)$.  Similarly, if $\delta = (\delta_\lambda)_{\lambda \in \Lambda} \in \Z^\Lambda$, and $\sigma = (\sigma_\lambda)_{\lambda \in \Lambda} \in \Z_2^\Lambda$, then $\{\delta,\sigma\} = \bigoplus_\lambda \{\delta_\lambda, \sigma_\lambda\}$ defines a grading shift automorphism of the categories $A\grm$ and $A\grp$.

\begin{defin}[Conjugate shifted adjointness] \label{def:twisted-adjoint}
  Suppose that, for $\lambda \in \Lambda$, $\psi_\lambda$ is an automorphism of graded superalgebras $A_\lambda \to A_\lambda$, $\delta_\lambda \in \Z$, and $\sigma_\lambda \in \Z_2$, and define $\Psi$ and $\{\delta,\sigma\}$ as above.  Then, given a tower of graded superalgebras $A$, we say that induction is \emph{conjugate shifted right adjoint} to restriction (and restriction is \emph{conjugate shifted left adjoint} to induction) with \emph{conjugation} $\Psi$ and \emph{shift} $\{\delta,\sigma\}$ if $\nabla$ is right adjoint to $\Psi^{\otimes 2} \{\delta^{\otimes 2}, \sigma^{\otimes 2}\} \Delta \Psi^{-1} \{-\delta,-\sigma\}$.
\end{defin}

There are some situations in which conjugate shifted adjointness is automatically satisfied.  For example, in Section~\ref{sec:frob-gr-alg}, we will see that this is the case for towers of Frobenius graded superalgebras.

\begin{defin}[Strong tower of graded superalgebras] \label{def:strong}
  Suppose $\chi = (\chi',\chi'')$ is a pair of biadditive maps $\chi',\chi'' \colon \Lambda \times \Lambda \to \Z$, $d \in \Z$, $\epsilon \in \Z_2$, $\delta = (\delta_\lambda)_{\lambda \in \Lambda} \in \Z^\Lambda$, and $\sigma = (\sigma_\lambda)_{\lambda \in \Lambda} \in \Z_2^\Lambda$.  We say that a tower of graded superalgebras $A$ is \emph{strong} with \emph{twist} $(\chi,d,\epsilon)$, \emph{conjugation} $\Psi$ and \emph{shift} $(\delta,\sigma)$ if the following two conditions are satisfied:
  \begin{description}[style=multiline, labelwidth=0.6cm]
    \item[\namedlabel{item:S1}{S1}] Induction is conjugate shifted right adjoint to restriction with conjugation $\Psi$ and shift $\{\delta,\sigma\}$, and there exists a biadditive map $\kappa \colon \Lambda \times \Lambda \to \Z$ such that the maps
      \[
        \kappa_\delta, \kappa_\sigma \colon \Lambda \times \Lambda \to \Z,\quad \kappa_\alpha(\lambda,\mu)=\alpha_{\lambda+\mu}-\alpha_\lambda-\alpha_\mu\quad \text{for } \alpha=\delta,\sigma,
      \]
      are of the form $\kappa_\delta=d\kappa$, $\kappa_\sigma=\epsilon\kappa$.

    \item[\namedlabel{item:S2}{S2}] We have an isomorphism of functors
      \begin{equation} \label{eq:strong-isom} \ts
        \Delta \nabla \cong \nabla^{\otimes 2} S_{23}\{ \chi \} \Delta^{\otimes 2},
      \end{equation}
      where $\{\chi\}$ is the endofunctor of $A\grm^{\otimes 4}$ given on $(A_{\lambda_1}\otimes A_{\lambda_2}\otimes A_{\lambda_3}\otimes A_{\lambda_4})\grm$ by $\{\chi\}(M_1\otimes M_2\otimes M_3\otimes M_4) = (M_1 \otimes M_2 \otimes M_3 \otimes M_4) \{d \chi'(\lambda_2, \lambda_3) + d\chi''(\lambda_1,\lambda_4), \epsilon \chi'(\lambda_2, \lambda_3) + \epsilon \chi''(\lambda_1,\lambda_4)\}$.
  \end{description}
  We say that the tower has \emph{trivial twist} if $\chi=0$ or $d=\epsilon=0$.
\end{defin}

\begin{defin}[Dualizing and compatible towers of graded superalgebras]
  We say that a tower of graded superalgebras $A$ is \emph{dualizing} if, under the operations~\eqref{eq:Groth-operations}, $\cK(A)$ and $\cG(A)$ are twisted Hopf algebras which are twisted dual, in the sense of Definition~\ref{def:dual-pair}, under the bilinear form~\eqref{eq:tower-pairing} (i.e.\ the perfect pairing~\eqref{eq:tower-pairing} is a twisted Hopf pairing). In addition, we say that $A$ is \emph{compatible} if the dual pair $(\cK(A),\cG(A))$ is compatible, as in Definition \ref{def:compatible-pair}.
\end{defin}

To simplify terminology, from now on we will just use the terms `strong tower', `dualizing tower' and `compatible tower', and will consider implied the specifier `of graded superalgebras'.  In the case that $\Lambda = \N$ and each $A_n$ lies in degree $(0,0)$, the following result was proved in~\cite[Prop.~3.7]{SY13}.

\begin{prop} \label{prop:dualizing-twisting-isom}
  Suppose $A$ is a strong tower with conjugation $\Psi$, twist $(\chi,d,\epsilon)$, and shift $\{\delta,\sigma\}$.  Then the following statements are equivalent.
  \begin{enumerate}
    \item \label{item:dual-twist} The tower $A$ is dualizing with a $(q^d\pi^\epsilon,0,\kappa)$-twisted Hopf pairing, where $\kappa$ is as in Definition \ref{def:strong}.

    \item \label{eq:twisted-coproduct-identity} We have $\Psi^{\otimes 2} \Delta \Psi^{-1} (P) \cong \Delta(P)$ for all $P \in A\grp$.

    \item \label{eq:twisted-coproduct-identity-classes} We have $\Psi^{\otimes 2} \Delta \Psi^{-1} = \Delta$ as maps $\cK(A) \to \cK(A) \otimes \cK(A)$.
  \end{enumerate}
  In particular, $A$ is dualizing if $\Psi = \id$ or, more generally, if $\Psi$ acts trivially on $\cK(A)$.
\end{prop}

\begin{proof}
  First assume that~\eqref{eq:twisted-coproduct-identity} holds.  The proof of the associativity of $\nabla$, the coassociativity of $\Delta$, and the first and third equations in Definition \ref{def:hopf-pairing} are almost identical to the arguments found in the proof of~\cite[Th.~3.6]{BL09}.  It then follows immediately from axiom~\ref{item:S2} in Definition~\ref{def:strong} that $\cG(A)$ is a $(q^d \pi^\epsilon,\chi)$-Hopf algebra and that $\cK(A)$ is a $(q^d \pi^\epsilon,-\chi)$-Hopf algebra, where $-\chi = (-\chi',-\chi'')$.    It remains to prove that
  \begin{equation} \label{eq:BLproofeq}
    \langle [P], \nabla ([M] \boxtimes [N]) \rangle = (q^d\pi^\epsilon)^{\kappa(|M|,|N|)} \langle \Delta ([P]), [M] \boxtimes [N] \rangle ,
  \end{equation}
  for all $M \in A_\lambda\grm$, $N \in A_\mu\grm$, $P \in A_{\lambda+\mu}\grp$.  However, under our assumptions, we have
  \begin{align*}
    \Hom_{A_{\lambda+\mu}} (P, \nabla(M \boxtimes N)) &\cong \Hom_{A_\lambda \otimes A_\mu} (\Psi^{\otimes 2}\{\delta^{\otimes 2},\sigma^{\otimes 2}\} \Delta \Psi^{-1}\{-\delta,-\sigma\}(P), M \boxtimes N) \\
    &\cong \Hom_{A_\lambda \otimes A_\mu} (\Psi^{\otimes 2} \Delta \Psi^{-1}(P)\{\delta_\lambda+\delta_\mu-\delta_{\lambda+\mu},\sigma_\lambda+\sigma_\mu-\sigma_{\lambda+\mu}\}, M \boxtimes N) \\
    & \cong \Hom_{A_\lambda \otimes A_\mu} (\Delta(P), M \boxtimes N) \{ d\kappa(\lambda,\mu), \epsilon\kappa(\lambda,\mu) \},
  \end{align*}
  which immediately implies~\eqref{eq:BLproofeq}.  Thus~\eqref{item:dual-twist} is true.

  Now suppose~\eqref{item:dual-twist} is true.  Then, for all $P \in A_{\lambda+\mu}\grp$, $M \in A_\lambda\grm$ and $N\in A_\mu\grm$ we have
  \[
    \langle \Psi^{\otimes 2} \Delta \Psi^{-1}([P]), [M] \boxtimes [N] \rangle =(q^d\pi^\epsilon)^{-\kappa(\lambda,\mu)} \langle [P], \nabla([M] \boxtimes [N]) \rangle = \langle \Delta([P]), [M] \boxtimes [N] \rangle,
  \]
  where the first equality holds by the assumption that induction is conjugate shifted right adjoint to restriction and the second equality holds by~\eqref{item:dual-twist}.  Then~\eqref{eq:twisted-coproduct-identity-classes} follows from the nondegeneracy of the bilinear form.

  The fact that~\eqref{eq:twisted-coproduct-identity} and~\eqref{eq:twisted-coproduct-identity-classes} are equivalent follows from the fact that every short exact sequence of projective supermodules splits.  Thus, for $P,Q \in A\grp$, we have $[P] = [Q]$ in $\cK(A)$ if and only if $P \cong Q$.
\end{proof}

%%%%%%%%%%%%%%%%%%%%%%%%%%%%%%%%%%%%%%%%%%%%%%%%%%%%
%
\section{Categorification of the twisted Heisenberg double and its Fock space} \label{sec:categorification}
%
%%%%%%%%%%%%%%%%%%%%%%%%%%%%%%%%%%%%%%%%%%%%%%%%%%%%

In this section we apply the constructions of Section~\ref{sec:general-construction} to the dual pair $(\cG(A),\cK(A))$ arising from a compatible tower $A$.  We will see that some natural subalgebras of the twisted Heisenberg double arise in this situation.  We then prove our main result (Theorem~\ref{theo:categorification}), the categorification of the Fock space representation of the twisted Heisenberg double.  As a corollary, we have a categorification of the twisted Heisenberg double itself (see Corollary~\ref{cor:Heis-double-cat}).  We fix a compatible tower $A$ (although many of the results of this section only require that $A$ be dualizing).

\begin{defin}[$\fh(A)$, $\cF(A)$, $\cG_\pj(A)$] \label{def:G-proj}
  We define
  \[
    \fh(A) = \fh(\cG(A),\cK(A)) \quad \text{and} \quad \cF(A) = \cF(\cG(A),\cK(A)).
  \]
  For each $\lambda \in \Lambda$, $A_\lambda\grp$ is a full subcategory of $A_\lambda\grm$.  The inclusion functor induces the \emph{Cartan map} $\cK(A) \to \cG(A)$, which is antilinear in the sense that it is linear in $\Z$ and $\pi$, but interchanges $q$ and $q^{-1}$.  Let $\cG_\pj(A)$ denote the image of the Cartan map.
\end{defin}

Let
\begin{equation} \label{eq:tower-notation}
  H^- = \cK(A),\ H^+ = \cG(A),\ H^+_\pj = \cG_\pj(A),\ \fh = \fh(A),\ \mathcal{F}=\mathcal{F}(A).
\end{equation}
To avoid confusion between $\cG(A)$ and $\cK(A)$, we will write $[M]_+$ to denote the class of a finitely-generated (possibly projective) graded $A_\lambda$-supermodule in $H^+$ and $[M]_-$ to denote the class of a finitely-generated graded projective $A_\lambda$-supermodule in $H^-$.  If $P \in A_\nu\grp$ and $N \in (A_{\lambda-\nu} \otimes A_\nu)\grm$, then we have a natural graded $A_{\lambda-\nu}$-supermodule structure on $\HOM_{A_\nu}(P,N)$ given by
\begin{equation} \label{eq:hom-restriction-action}
  a \cdot f = \pL{(a \otimes 1)} \circ f,\quad a \in A_{\lambda-\nu},\ f \in \Hom_{A_\nu}(P,N).
\end{equation}

\begin{lem} \label{lem:adjoint-action-explicit}
  If $\lambda,\nu \in \Lambda$, $P \in A_\nu\grp$ and $N \in A_\lambda\grm$, then we have
  \[
    [P]_- \cdot [N]_+ =
    \begin{cases}
      0 & \text{if } \nu \not \leq  \lambda, \\
      [\HOM_{A_\nu}(P,\Res^{A_\lambda}_{A_{\lambda-\nu} \otimes A_\nu} N)]_+ & \text{if } \nu \le \lambda,
    \end{cases}
  \]
  where $\nu \le \lambda$ if and only if there exists $\mu \in \Lambda$ such that $\nu + \mu = \lambda$.  Here $\cdot$ denotes the action of $\fh$ on $\cF$ and $\HOM_{A_\nu}(P,\Res^{A_\lambda}_{A_{\lambda-\nu} \otimes A_\nu} N)$ is viewed as an $A_{\lambda-\nu}$-supermodule as in~\eqref{eq:hom-restriction-action}.
\end{lem}

\begin{proof}
  The proof is analogous to that of \cite[Lem.~3.10]{SY13}.
  \details{
  The case $\nu\not\leq\lambda$ follows immediately from the fact that $H^+_{\lambda-\nu}=0$ if $\nu\not\leq\lambda$.  Assume $\nu \le \lambda$.  For $R \in A_{\lambda-\nu}\grp$, we have
  \begin{align*}
    \langle [R]_-, [P]_- \cdot [N]_+ \rangle &= \langle [R]_- [P]_-, [N]_+ \rangle \\
    &= \langle \nabla([R]_- \boxtimes [P]_-), [N]_+ \rangle \\
    &= \langle [R]_- \boxtimes [P]_-, \Delta ([N]_+) \rangle \\
    &= \qdim_\F\HOM_{A_{\lambda-\nu} \otimes A_\nu} (R \boxtimes P, \Res^{A_\lambda}_{A_{\lambda-\nu} \otimes A_\nu} N) \\
    &= \qdim_\F \HOM_{A_{\lambda-\nu}} (R, \HOM_{A_\nu}(P,\Res^{A_\lambda}_{A_{\lambda-\nu} \otimes A_\nu} N)) \\
    &= \langle [R]_- ,[\HOM_{A_\nu}(P,\Res^{A_\lambda}_{A_{\lambda-\nu} \otimes A_\nu} N)]_+ \rangle.
  \end{align*}
  The result then follows from the nondegeneracy of the bilinear form.
  }
\end{proof}

\begin{prop} \label{prop:Hpj-invariance}
  We have that $H^+_\pj$ is a subalgebra of $H^+$ that is invariant under the left regular action of $H^-$.
\end{prop}

\begin{proof}
  The proof is analogous to that of \cite[Prop.~3.12]{SY13}.
  \details{
    First we prove an analogue of~\cite[Lem.~3.11]{SY13}. Suppose $\kk$ is a commutative ring and $R$ and $S$ are graded $\kk$-superalgebras.  Furthermore, suppose that $P\in S\grp$ and $Q\in (R\otimes S)\grp$.  There exist $s,t \in \N$, a graded $S$-supermodule $P'$, and a graded $(R \otimes S)$-supermodule $Q'$ such that $P \oplus P' \cong \bigoplus_{i=1}^s S\{n_i\}$ and $Q \oplus Q' \cong \bigoplus_{j=1}^t(R \otimes S)\{m_j\}$, for some integers $n_i, m_j$.  Then we have
    \begin{multline*} \ts
      \HOM_S(P,Q) \oplus \HOM_S(P,Q') \oplus \HOM_S(P', \bigoplus_{j=1}^t(R \otimes S)\{m_j\}) \\ \ts
      \cong \HOM_S (\bigoplus_{j=1}^s S\{n_j\}, \bigoplus_{j=1}^t(R \otimes S)\{m_j\}) \cong \bigoplus_{i,j} \HOM_S(S\{n_i\},(R \otimes S)\{m_j\}) \\ \ts
      \cong \bigoplus_{i,j} \HOM_S(S,R \otimes S)\{m_j-n_i\} \cong \bigoplus_{i,j} R\{m_j-n_i\}.
    \end{multline*}
    Thus $\HOM_S(P,Q)$ is a projective graded $R$-supermodule.

    Now we can prove the lemma.  Assume $P \in A_\lambda\grp$ and $Q \in A_\mu\grp$.  As in the proof of \cite[Prop.~3.2]{BL09}, we have that $\Ind^{A_{\lambda+\mu}}_{A_\lambda \otimes A_\mu} (P \boxtimes Q)$ is a projective $A_{\lambda+\mu}$-supermodule.  It follows that $H^+_\pj$ is a subalgebra of $H^+$.

    By Lemma~\ref{lem:adjoint-action-explicit}, it remains to show that
    \[
      \HOM_{A_\lambda}(P, \Res^{A_\mu}_{A_{\mu-\lambda} \otimes A_\lambda} Q) \in A_{\mu-\lambda}\grp.
    \]
    Again, as in the proof of \cite[Prop.~3.2]{BL09}, we have that $\Res^{A_\mu}_{A_{\mu-\lambda} \otimes A_\lambda} Q$ is a projective $(A_{\mu-\lambda} \otimes A_\lambda)$-supermodule.  The result then follows from the above analogue of~\cite[Lem.~3.11]{SY13}.
  }
\end{proof}

\begin{defin}[The projective twisted Heisenberg double ${\mathfrak{h}_\pj}(A)$] \label{def:p}
  By Proposition~\ref{prop:Hpj-invariance}, ${\mathfrak{h}_\pj}={\mathfrak{h}_\pj}(A) := H^+_\pj \# H^-$ is a subalgebra of $\fh$.  In other words, ${\mathfrak{h}_\pj}$ is the subalgebra of $\fh$ generated by $H^+_\pj$  and $H^-$ (viewing the latter two as $\kk$-submodules of $\fh$ as in Definition~\ref{def:h}).  We call ${\mathfrak{h}_\pj}$ the \emph{projective twisted Heisenberg double} associated to $A$.
\end{defin}

\begin{defin}[Fock space $\cF_\pj(A)$ of ${\mathfrak{h}_\pj}$] \label{def:p-Fock-space}
  By Proposition~\ref{prop:Hpj-invariance}, the algebra ${\mathfrak{h}_\pj}$ acts on $H^+_\pj$.  We call this the \emph{lowest weight Fock space representation} of ${\mathfrak{h}_\pj}$ and denote it by $\cF_\pj = \cF_\pj(A)$.  Note that this representation is generated by the lowest weight vacuum vector $1 \in H^+_\pj$.
\end{defin}

Recall the direct sums of categories
\[ \ts
  A\grm = \bigoplus_{\lambda \in \Lambda} A_\lambda\grm,\quad A\grp = \bigoplus_{\lambda \in \Lambda} A_\lambda\grp.
\]
For each $M \in A_\mu\grm$, $\mu \in \Lambda$, define the functor $\Ind_M \colon A\grm \to A\grm$ by
\[
  \Ind_M(N) = \Ind^{A_{\mu+\lambda}}_{A_\mu \otimes A_\lambda} (M \boxtimes N) \in A_{\mu+\lambda}\grm,\quad N \in A_\lambda\md,\ \lambda \in \Lambda.
\]
For each $P \in A_\nu\grp$, $\nu \in \Lambda$, define the functor $\Res_P \colon A\grm \to A\grm$ by
\[
  \Res_P(N) = \HOM_{A_\nu}(P, \Res^{A_\lambda}_{A_{\lambda-\nu} \otimes A_\nu} N) \in A_{\lambda-\nu}\grm,\quad N \in A_\lambda\grm,\ \lambda \in \Lambda,
\]
where $\Res_P(N)$ is interpreted to be the zero object of $A\grm$ if $\lambda\not\geq\nu$.  Since
\[
  \Ind_P(A\grp) \subseteq A\grp, \quad \Res_P(A\grp) \subseteq A\grp, \quad \text{for all } P \in A\grp,
\]
we have the induced functors $\Ind_P, \Res_P \colon A\grp \to A\grp$ for $P \in A\grp$.

Because the functors $\Ind_M$ and $\Res_P$ are exact for all $M \in A\grm$ and $P \in A\grp$, they induce endomorphisms $[\Ind_M]$ and $[\Res_P]$ of $\cG(A)$.  Similarly, $\Ind_P$ and $\Res_P$ induce endomorphisms $[\Ind_P]$ and $[\Res_P]$ of $\cG_\pj(A)$ for all $P \in A\grp$.

\begin{prop} \label{prop:cat-Groth-group}
  Suppose $A$ is a compatible tower.
  \begin{enumerate}
    \item \label{theo-item:weak-cat} For all $M,N \in A\grm$ and $P \in A\grp$, we have
        \[
          ([M] \# [P])([N]) = [\Ind_M] \circ [\Res_P] ([N]) = [\Ind_M \circ \Res_P (N)] \in \cG(A).
        \]

    \item \label{theo-item:weak-cat-proj} For all $Q,P,R \in A\grp$, we have
        \[
          ([Q] \# [P])([R]) = [\Ind_Q] \circ [\Res_P] ([R]) = [\Ind_Q \circ \Res_P (R)] \in \cG_\pj(A).
        \]
  \end{enumerate}
\end{prop}

\begin{proof}
  This follows from the definition of the multiplication in $\cG(A)$ and Lemma~\ref{lem:adjoint-action-explicit}.
\end{proof}

Part~\eqref{theo-item:weak-cat} (resp.\ part~\eqref{theo-item:weak-cat-proj}) of Proposition~\ref{prop:cat-Groth-group} shows how the action of $\fh$ on $\cF$ (resp.\ $\fh_\pj$ on $\cF_\pj$) is induced by functors on $\bigoplus_{\lambda \in \Lambda} A_\lambda\grm$ (resp.\ $\bigoplus_{\lambda \in \Lambda}A_\lambda\grp$).  Typically a \emph{categorification} of a representation consists of isomorphisms of such functors which lift the algebra relations.  As we now describe, this can be done if the tower is also strong.  For the remainder of this section, we fix a strong compatible tower $A$ with twist $(\chi,d,\epsilon)$, shift $(\delta,\sigma)$ and biadditive map $\kappa$ as in Definition~\ref{def:strong}.

Since the natural action of $\fh$ on $H^+$ is faithful by~\cite[Th.~4.3(d)]{RS14}, the algebra structure on $\fh$ is uniquely determined by the fact that $H^\pm$ are subalgebras and by~\eqref{eq:H-general-relation}, which gives the commutation relation between $H^+$ and $H^-$.  Now, recalling that $(\gamma',\gamma'')=(0,\kappa)$, \eqref{eq:H-general-relation} is equivalent to the following equalities in $\End H^+$:
\begin{equation} \label{eq:fh-key-relation-action}
  \pR{x}^* \circ \pL{a} = \nabla \left( {\Delta(x)}^\dagger (a \otimes -)\right),\quad x \in H^-,\ a \in H^+,
\end{equation}
where, for homogeneous $x, y \in H^-$, we define the operator $(x\otimes y)^\dagger \in \End(H^+\otimes H^+)$ by
\[
  (x\otimes y)^\dagger (a\otimes b) = (q^d\pi^\epsilon)^{\kappa(|x|,|y|) + \chi'(|x|, |b| - |y|) + \chi''(|a|-|x|,|y|)}\pR{x}^*(a)\otimes \pR{y}^*(b),
\]
and extend by linearity.  For $Q \in (A_\nu \otimes A_\rho)\grp$, $M \in A_\mu\grm$, and $N \in A_\lambda\grm$, define
\begin{multline*}
  \Res^\dagger_Q(M \boxtimes N) := \HOM_{A_\nu \otimes A_\rho} \left( Q, S_{23} \left( \Res^{A_\mu}_{A_{\mu-\nu} \otimes A_\nu} M \boxtimes \Res^{A_\lambda}_{A_{\lambda-\rho} \otimes A_\rho} N \right) \right) \\
  \{d\kappa(\nu,\rho) + d\chi'(\nu, \lambda - \rho) + d \chi''(\mu-\nu,\rho),\epsilon\kappa(\nu,\rho)+ \epsilon \chi'(\nu, \lambda - \rho) +\epsilon\chi''(\mu-\nu,\rho)\}.
\end{multline*}

\begin{theo} \label{theo:categorification}
  Suppose that $A$ is a strong compatible tower.  Then we have the following isomorphisms of functors for all $M,N \in A\grm$ and $P,Q \in A\grp$.
  \begin{gather}
    \Ind_M \circ \Ind_N \cong \Ind_{\nabla(M \boxtimes N)}, \label{cat-eq:ind} \\
    \Res_P \circ \Res_Q \cong \Res_{\nabla(P \boxtimes Q)}, \label{cat-eq:res} \\ \ts
    \Res_P \circ \Ind_M \cong \nabla \Res^\dagger_{\Delta(P)}(M \boxtimes -). \label{cat-eq:cross}
  \end{gather}
  In particular, the above yields a categorification of the lowest weight Fock space representations of $\fh(A)$ and $\fh_\pj(A)$.
\end{theo}

\begin{proof}
  The proofs of~\eqref{cat-eq:ind} and~\eqref{cat-eq:res} are identical to the proofs of the analogous statements in~\cite[Th.~3.18]{SY13} and are therefore omitted.
  \details{
  Suppose $M \in A_\mu\grm$, $N \in A_\lambda\grm$, $P \in A_\nu\grp$, $Q \in A_\rho\grp$, and $L \in A_\sigma\grm$.  Then we have
  \begin{align*}
    \Ind_M \circ \Ind_N (L) &= \Ind^{A_{\mu+\lambda+\sigma}}_{A_\mu \otimes A_{\lambda + \sigma}} \left( M \boxtimes \Ind^{A_{\lambda + \sigma}}_{A_\lambda \otimes A_\sigma} (N \boxtimes L) \right) \\
    &\cong \Ind^{A_{\mu+\sigma}}_{A_\mu \otimes A_{\lambda + \sigma}} \Ind^{A_\mu \otimes A_{\lambda + \sigma}}_{A_\mu \otimes A_\lambda \otimes A_\sigma} (M \boxtimes N \boxtimes L) \\
    &\cong \Ind^{A_{\mu+\lambda+\sigma}}_{A_\mu \otimes A_\lambda \otimes A_\sigma} (M \boxtimes N \boxtimes L) \\
    &\cong \Ind^{A_{\mu+\lambda+\sigma}}_{A_{\mu+\lambda} \otimes A_\sigma} \Ind^{A_{\mu+\lambda} \otimes A_\sigma}_{A_\mu \otimes A_\lambda \otimes A_\sigma} (M \boxtimes N \boxtimes L) \\
    &\cong \Ind^{A_{\mu+\lambda+\sigma}}_{A_{\mu+\lambda} \otimes A_\sigma} \left( \Ind^{A_{\mu+\lambda}}_{A_\mu \otimes A_\lambda} (M \boxtimes N) \boxtimes L \right) \\
    &\cong \Ind_{\nabla(M \boxtimes N)} L.
  \end{align*}
  Since each of the above isomorphisms is natural in $L$, this proves~\eqref{cat-eq:ind}.

  Similarly, we have
  \begin{align*}
    \Res_P \circ \Res_Q (L) &= \HOM_{A_\nu} (P, \Res^{A_{\sigma-\rho}}_{A_{\sigma-\rho-\nu} \otimes A_\nu} \HOM_{A_\rho} (Q, \Res^{A_\sigma}_{A_{\sigma-\rho} \otimes A_\rho} L)) \\
    &\cong \HOM_{A_\nu} (P, \HOM_{A_\rho} (Q, \Res^{A_\sigma}_{A_{\sigma-\nu-\rho} \otimes A_\nu \otimes A_\rho} L)) \\
    &\cong \HOM_{A_\nu \otimes A_\rho} (P \boxtimes Q, \Res^{A_\sigma}_{A_{\sigma-\nu-\rho} \otimes A_\nu \otimes A_\rho} L) \\
    &\cong \HOM_{A_\nu \otimes A_\rho} (P \boxtimes Q, \Res^{A_{\sigma-\nu-\rho} \otimes A_{\nu+\rho}}_{A_{\sigma-\nu-\rho} \otimes A_\nu \otimes A_\rho} \Res^{A_\sigma}_{A_{\sigma-\nu-\rho} \otimes A_{\nu+\rho}} L) \\
    &\cong \HOM_{A_{\nu+\rho}} (\Ind^{A_{\nu+\rho}}_{A_\nu \otimes A_\rho} (P \boxtimes Q),\Res^{A_\sigma}_{A_{\sigma-\nu-\rho} \otimes A_{\nu+\rho}} L) \\
    &\cong \Res_{\nabla(P \boxtimes Q)} L.
  \end{align*}
  Since each of the above isomorphisms is natural in $L$, this proves~\eqref{cat-eq:res}.
  }
  It remains to prove~\eqref{cat-eq:cross}.  For simplicity, in what follows we will use the notation $\Res^{\sigma+\nu}_{\sigma,\nu}$ for $\Res^{A_{\sigma+\nu}}_{A_\sigma \otimes A_\nu}$ and an analogous notation for induction functors.  We will also use a similar notation if the restriction or induction involves any number of tensor factors.  We have, for $P \in A_\nu\grp$, $M \in A_\mu\grm$, and $L \in A_\rho\grm$,
  \begin{align*}
    &\Res_P \circ \Ind_M (L) = \HOM_{A_\nu} \left(P, \Res^{\mu+\rho}_{\mu+\rho-\nu,\nu} \Ind^{\mu+\rho}_{\mu,\rho} (M \boxtimes L) \right) \\
    &\cong \ts \HOM_{A_\nu} \left( P, \bigoplus\limits_{\alpha+\beta=\nu} \Ind^{\mu+\rho-\nu, \nu}_{\mu-\alpha,\rho-\beta,\alpha,\beta} S_{23} \Res^{\mu,\rho}_{\mu-\alpha,\alpha,\rho-\beta,\beta} (M \boxtimes L) \right) \{\zeta',\zeta''\} \\
    &\cong \ts \HOM_{A_\nu} \left(P, \bigoplus\limits_{\alpha+\beta=\nu} \Ind^{\mu+\rho-\nu,\nu}_{\mu-\alpha,\rho-\beta,\nu} \Ind^{\mu-\alpha,\rho-\beta,\nu}_{\mu-\alpha,\rho-\beta,\alpha,\beta} S_{23} \Res^{\mu,\rho}_{\mu-\alpha,\alpha,\rho-\beta,\beta} (M \boxtimes L)\right) \{\zeta',\zeta''\} \\
    &\cong \ts \bigoplus\limits_{\alpha+\beta=\nu} \Ind^{\mu+\rho-\nu}_{\mu-\alpha,\rho-\beta} \HOM_{A_\nu} \left( P, \Ind^{\mu-\alpha,\rho-\beta,\nu}_{\mu-\alpha,\rho-\beta,\alpha,\beta} S_{23} \Res^{\mu,\rho}_{\mu-\alpha,\alpha,\rho-\beta,\beta} (M \boxtimes L)\right) \{\zeta',\zeta''\} \\
    &\cong \ts \bigoplus\limits_{\alpha+\beta=\nu} \Ind^{\mu+\rho-\nu}_{\mu-\alpha,\rho-\beta} \HOM_{A_\alpha \otimes A_\beta} \big(\Res^{\nu}_{\alpha,\beta} (P)\{-\kappa_\delta(\alpha,\beta),-\kappa_\sigma(\alpha,\beta)\},
    S_{23} \Res^{\mu,\rho}_{\mu-\alpha,\alpha,\rho-\beta,\beta} (M \boxtimes L)\big) \{\zeta',\zeta''\} \\
    &\cong \ts \bigoplus\limits_{\alpha+\beta=\nu} \Ind^{\mu+\rho-\nu}_{\mu-\alpha,\rho-\beta} \HOM_{A_\alpha \otimes A_\beta} \big( \Res^{\nu}_{\alpha,\beta} (P)\{-d\kappa(\alpha,\beta),-\epsilon\kappa(\alpha,\beta)\},
    S_{23} \Res^{\mu,\rho}_{\mu-\alpha,\alpha,\rho-\beta,\beta} (M \boxtimes L)\big) \{\zeta',\zeta''\} \\
    &\cong \nabla \Res^\dagger_{\Delta(P)} (M \boxtimes L),
  \end{align*}
  where
  \[
    \zeta' = d\chi'(\alpha,\rho-\beta) + d\chi''(\mu-\alpha,\beta),\quad \zeta'' = \epsilon\chi'(\alpha,\rho-\beta) + \epsilon\chi''(\mu-\alpha,\beta).
  \]
  The first isomorphism above follows from~\eqref{eq:strong-isom}.  In the fourth isomorphism, we use the fact that induction is conjugate shifted right adjoint to restriction and that the tower is dualizing.  Since all of the above isomorphisms are natural in $L$,
  this proves~\eqref{cat-eq:cross}.

  The final assertion of the theorem follows as explained in the paragraph preceding the statement of the theorem.  In particular, the isomorphisms~\eqref{cat-eq:ind} and~\eqref{cat-eq:res} categorify the multiplication in $\cG(A)$ and $\cK(A)$, respectively, and the isomorphism~\eqref{cat-eq:cross} categorifies the relation~\eqref{eq:fh-key-relation-action}.
\end{proof}

\begin{cor} \label{cor:Heis-double-cat}
  Let $\mathcal{H}$ be the full subcategory of $\End(A\grm)$ generated by $\Ind_M$, $M \in A\grm$, and $\Res_P$, $P \in A\grp$, under composition and degree shift.  Then $\cK_0(\mathcal{H}) \cong \fh(A)$.
\end{cor}

\begin{proof}
  By Theorem~\ref{theo:categorification}, we have a surjective map
  \[
    \fh(A) \twoheadrightarrow \cK_0(\mathcal{H}),\quad [M] \# [P] \mapsto [\Ind_M] \circ [\Res_P],\quad M \in A\grm,\ P \in A\grp.
  \]
  By~\cite[Th.~4.3(a)]{RS14}, the kernel of this map must be zero, hence the map is also injective.
\end{proof}

Note that the categorifications of Theorem~\ref{theo:categorification} and~Corollary~\ref{cor:Heis-double-cat} do not rely on a particular presentation of the twisted Heisenberg double $\fh(A)$.

%%%%%%%%%%%%%%%%%%%%%%%%%%%%%%%%%%%%%%%%%%%%%%%%%%%%
%
\section{Frobenius graded superalgebras} \label{sec:frob-gr-alg}
%
%%%%%%%%%%%%%%%%%%%%%%%%%%%%%%%%%%%%%%%%%%%%%%%%%%%%

Many existing constructions in categorification involve Frobenius (super)algebras.  In this section, we show that towers of Frobenius graded superalgebras automatically satisfy some of the axioms of a strong tower.  In particular, induction is always conjugate shifted right adjoint to restriction (see Proposition~\ref{prop:Frobenius-twisted}).  Our definition of a Frobenius superalgebra is more general that the definition typically appearing in the literature (those definitions require $\sigma=0$ in Definition~\ref{def:frobenius} below).  In addition, we consider the graded version of these algebras.

\begin{defin}[Frobenius graded superalgebra] \label{def:frobenius}
  Suppose $\delta \in \Z$ and $\sigma \in \Z_2$.  We say that a finite-dimensional graded superalgebra $B$ is a \emph{Frobenius graded superalgebra of degree $(-\delta,\sigma)$} if one of the following three equivalent conditions holds:
  \begin{enumerate}
    \item \label{def-item:Frob-dual} There is an isomorphism (homogeneous of degree zero) of graded left $B$-supermodules
        \[
      	 \varphi \colon B \to B^\vee\{\delta,\sigma\} = \HOM_\F(B,\F)\{\delta,\sigma\} = \HOM_\F(B,\F\{\delta,\sigma\}).
        \]
        Here $B^\vee$ is a graded left $B$-supermodule as in \eqref{eq:left-hom-action}.

    \item \label{def-item:Frob-form} There exists a nondegenerate invariant graded $\F$-bilinear form $(-,-) \colon B\times B\to \F\{\delta,\sigma\}$.  By \emph{invariant}, we mean that
        \[
          (ab,c)=(a,bc) \quad \text{for all } a,b,c\in B.
        \]
        By \emph{graded}, we mean that $(B_{\lambda,\epsilon}, B_{\lambda',\epsilon'}) \subseteq \F\{\delta,\sigma\}_{\lambda+\lambda',\epsilon+\epsilon'}$.  In other words, $B_{\lambda,\epsilon} \perp B_{\lambda',\epsilon'}$ unless $\lambda + \lambda' = \delta$ and $\epsilon + \epsilon' = \sigma$, since we view $\F$ as lying in degree zero.

    \item \label{def-item:Frob-functional} There exists an $\F$-linear graded (i.e.\ homogeneous of degree zero) map $\tr \colon B \to \F\{\delta,\sigma\}$, called the \emph{trace map}, such that the kernel of $\tr$ contains no nonzero left ideals of $B$.
  \end{enumerate}
\end{defin}

The relationship between the various structures in Definition~\ref{def:frobenius} is as follows:  For $a,b \in B$, we have
\[
  \varphi(b)(a) = (-1)^{\bar a \bar b} (a,b),\quad \tr(b) = (b,1) = (1,b),\quad (a,b) = \tr(ab).
\]
The proof that the conditions in Definition~\ref{def:frobenius} are equivalent is a straightforward generalization of the analogous fact for Frobenius algebras.

\details{
  Suppose~\eqref{def-item:Frob-dual} is true.  For each $b\in B$, define $\varphi_b:=\varphi(b)\in\HOM_\F(B,\F)\{\delta,\sigma\}$. Then define a bilinear form on $B$ by $(a,b):=(-1)^{\bar{a}\bar{b}}\varphi_b(a)$ for $a,b\in B$. This bilinear form is nondegenerate because $\varphi$ is injective. Since $\varphi$ is an isomorphism of graded left $B$-supermodules, we have that
  \[
    \varphi_{bc}=\varphi_{b\cdot c} = b\cdot \varphi_c= (-1)^{\bar{b}\bar{c}}\varphi_c\circ\pr{b},\quad b,c \in B.
  \]
  Hence, for $a,b,c\in B$,
  \begin{align*}
    (a,bc)&=(-1)^{\bar{a}\bar{b}+\bar{a}\bar{c}}\varphi_{bc}(a)\\
    &=(-1)^{\bar{a}\bar{b}+\bar{a}\bar{c}+\bar{b}\bar{c}}(\varphi_c\circ\pr{b})(a) \\ &=(-1)^{\bar{a}\bar{b}+\bar{a}\bar{c}+\bar{b}\bar{c}}\varphi_c((-1)^{\bar{a}\bar{b}}ab) \\
    &=(-1)^{\bar{a}\bar{c}+\bar{b}\bar{c}}\varphi_c(ab)\\
    &=(ab,c)
 \end{align*}
  The fact that $\varphi$ is homogeneous of degree zero implies that $(-,-)$ is graded.

  Now assume that~\eqref{def-item:Frob-form} holds.  Define, for homogeneous $b \in B$, the operator $\varphi(b)\in\HOM_\F(B,\F\{\delta,\sigma\}) = \HOM_\F(B,\F)\{\delta,\sigma\}$ by the formula
  \[
    \varphi(b)(a)=(-1)^{\bar{a}\bar{b}}(a,b)\quad\text{ for all homogeneous }a\in B.
  \]
  Then, by reversing the arguments above, we get that~\eqref{def-item:Frob-dual} holds. The invariance of the bilinear form and the fact that $(-,-)$ is graded imply that $\varphi$ is a degree zero map of graded supermodules. The nondegeneracy of the bilinear form implies that $\varphi$ is injective, hence it is a bijection because both sides have the same dimension as vector spaces over $\F$.  Thus~\eqref{def-item:Frob-dual} is true.

  It remains to prove that~\eqref{def-item:Frob-form} and~\eqref{def-item:Frob-functional} are equivalent. To a nondegenerate invariant $\F$-bilinear form $(-,-)$, we associate the linear functional $\tr$, given by $\tr(b)=(b,1)=(1,b)$, whose kernel contains no nonzero left ideals.  Similarly, to the linear functional $\tr$ with this property, we associate the nondegenerate invariant $\F$-bilinear form $(-,-)$ given by $(a,b) = \tr(ab)$.  (See, for example,~\cite[Lem.~2.2.4]{Koc04}.)  It is straightforward to check that $\tr$ is graded if and only if $(-,-)$ is.
}

\begin{rem}
  Note that the choice of $\delta$ in Definition~\ref{def:frobenius} is unique, since it is the maximal value of $\delta$ such that $B_{\delta,\epsilon} \ne 0$ for some $\epsilon \in \Z_2$.  However, the $\Z_2$-degree ($\sigma$ in Definition~\ref{def:frobenius}) of $B$ is not necessarily unique.  In what follows, when we say that $B$ is a Frobenius graded superalgebra of degree $(-\delta,\sigma)$, we assume that we have fixed a choice of trace map (equivalently, isomorphism $\varphi$ or bilinear form) satisfying the conditions in Definition~\ref{def:frobenius} for that degree.
\end{rem}

\begin{lem}[Nakayama automorphism] \label{lem:Nakayama-def}
  If $B$ is a Frobenius graded superalgebra, then there exists an automorphism $\psi \colon B \to B$ of graded superalgebras such that $(a,b)=(-1)^{\bar{a}\bar{b}}(b,\psi(a))$ for all homogeneous $a,b\in B$.  This automorphism is called the \emph{Nakayama automorphism} of $B$.
\end{lem}

The following proof is a rather straightforward generalization of the proof of the existence of the Nakayama automorphism in the non-graded, non-super setting (see, for example, \cite[\S16E]{Lam99}).

\begin{proof}
  For a fixed $a \in B$, the map $b \mapsto (-1)^{\bar a \bar b}(a,b)$ is an $\F$-linear functional on $B$.  Thus, it is of the form $b \mapsto (b,\psi(a))$ for some $\psi(a) \in B$.  It is straightforward to verify that the map $\psi \colon B \to B$ thus defined is an automorphism of $\F$-vector spaces.  Now, suppose $B$ is of degree $(-\delta,\sigma)$.  For $\lambda,\lambda' \in \Z$ and $\epsilon,\epsilon' \in \Z_2$, we have
  \[
    (B_{\lambda',\epsilon'}, \psi(B_{\lambda,\epsilon})) = (B_{\lambda,\epsilon}, B_{\lambda',\epsilon'}) = 0 \quad \text{if } \lambda + \lambda' \ne \delta \text{ or } \epsilon + \epsilon' \ne \sigma.
  \]
  Thus $\psi(B_{\lambda,\epsilon}) \subseteq B_{\lambda,\epsilon}$, and so $\psi$ is a map of $(\Z,\Z_2)$-graded vector spaces.  Finally, for all homogeneous $a,b,c\in B$, we have
  \begin{multline*} (b,\psi(ac))=(-1)^{\bar{a}\bar{b}+\bar{c}\bar{b}}(ac,b)
    =(-1)^{\bar{a}\bar{b}+\bar{c}\bar{b}}(a,cb)
    =(-1)^{\bar{c}\bar{b}+\bar{a}\bar{c}}(cb,\psi(a)) \\
    =(-1)^{\bar{c}\bar{b}+\bar{a}\bar{c}}(c,b\psi(a))
    =(b\psi(a),\psi(c)) = (b,\psi(a)\psi(c)).
  \end{multline*}
  Hence $\psi(ac)=\psi(a)\psi(c)$ by the nondegeneracy of the bilinear form, and so $\psi$ is an automorphism of graded superalgebras.
\end{proof}

\begin{lem} \label{lem:Frobenius-alg-tensor-prod}
  For $i=1,2$, let $B_i$ be a Frobenius graded superalgebra of degree $(-\delta_i,\sigma_i)$, with trace map $\tr_i$ and Nakayama automorphism $\psi_i$.  Then $B_1 \otimes B_2$ is a Frobenius graded superalgebra of degree $(-\delta_1 - \delta_2, \sigma_1 + \sigma_2)$, with trace map $\tr_1 \otimes \tr_2$ and Nakayama automorphism $\psi_1 \otimes \psi_2$. The invariant bilinear form on  $B_1\otimes B_2$ is defined in terms of the invariant bilinear forms on $B_1$ and $B_2$ by
  \[
    (b_1\otimes b_2,c_1\otimes c_2):=(-1)^{\bar b_2 \bar c_1}(b_1,c_1)(b_2,c_2),
  \]
  for all homogeneous $b_1,c_1 \in B_1$, $b_2,c_2 \in B_2$.
\end{lem}

\begin{proof}
  The straightforward proof of this lemma is left to the reader.
  \details{
    Define the bilinear form as above.  It is clearly graded of degree $(-\delta_1 - \delta_2, \sigma_1 + \sigma_2)$. To check that it is invariant, we have, for all homogeneous $a_1,b_1,c_1\in B_1$ and $a_2,b_2,c_2\in B_2$,
    \begin{align*}
      ((b_1\otimes b_2)(a_1\otimes a_2),c_1\otimes c_2) &= (-1)^{\bar b_2 \bar a_1}(b_1a_1\otimes b_2a_2,c_1\otimes c_2) \\
      &=(-1)^{\bar b_2 \bar a_1 + \bar b_2 \bar c_1 + \bar a_2 \bar c_1}(b_1a_1,c_1)(b_2a_2,c_2) \\
      &=(-1)^{\bar b_2 \bar a_1 + \bar b_2 \bar c_1 + \bar a_2 \bar c_1}(b_1,a_1c_1)(b_2,a_2c_2) \\
      &=(-1)^{\bar a_2 \bar c_1}(b_1\otimes b_2,a_1c_1\otimes a_2c_2) \\
      &=(b_1\otimes b_2,(a_1\otimes a_2)(c_1\otimes c_2)).
    \end{align*}
    Now, the trace form on $B_1\otimes B_2$ is defined by
    \[
      \tr_{12}(b_1\otimes b_2) = (b_1\otimes b_2,1_{B_1}\otimes 1_{B_2}) = (-1)^{\bar b_2 \cdot 0}(b_1,1_{B_1})(b_2,1_{B_2}) = \tr_1(b_1)\tr_2(b_2).
    \]
    Hence $\tr_{12}=\tr_1\otimes\tr_2$.
    For the Nakayama automorphism, we have, for all homogeneous $a_1,b_1\in B_1$ and $a_2,b_2\in B_2$,
    \begin{align*}
      (a_1\otimes a_2,b_1\otimes b_2) &=(-1)^{\bar a_2 \bar b_1}(a_1,b_1)(a_2,b_2) \\
      &=(-1)^{\bar a_2 \bar b_1 + \bar a_1 \bar b_1 + \bar a_2 \bar b_2}(b_1,\psi_1(a_1))(b_2,\psi_2(a_2)) \\
      &=(-1)^{\bar a_2 \bar b_1 + \bar a_1 \bar b_1 + \bar a_2 \bar b_2 + \bar a_1 \bar b_2}(b_1\otimes b_2,\psi_1(a_1)\otimes \psi_2(a_2)) \\
      &=(-1)^{(\bar a_1 + \bar a_2)(\bar b_1 + \bar b_2)}(b_1\otimes b_2,\psi_1(a_1)\otimes \psi_2(a_2)).
    \end{align*}
    It follows that $\psi_1\otimes\psi_2$ is the Nakayama automorphism of $B_1\otimes B_2$.
  }
\end{proof}

\begin{lem} \label{lem:dual-module-isom}
  If $B$ is a Frobenius graded superalgebra of degree $(-\delta,\sigma)$, then we have an isomorphism of graded $(B,B)$-superbimodules
  \[
    B^\psi \cong B^\vee\{\delta,\sigma\}.
  \]
\end{lem}

\begin{proof}
  By Definition \ref{def:frobenius}\eqref{def-item:Frob-dual} we have an isomorphism $\varphi \colon B \to B^\vee\{\delta,\sigma\}$ of graded left $B$-supermodules.  It remains to show that, under the isomorphism $\varphi$, the right $B$-action on $B^\vee\{\delta,\sigma\}$ defined by~\eqref{eq:left-module-dual-right-action} corresponds to the right $B$-action on $B$ twisted by the Nakayama automorphism.   For homogeneous $a,b,c \in B$, we have
  \begin{multline*}
    (\varphi(b) \cdot a)(c) = (\varphi(b) \circ \pL{a})(c)
    = \varphi(b)(ac)
    = (-1)^{\bar a \bar b + \bar b \bar c} (ac,b)
    = (-1)^{\bar a \bar b + \bar b \bar c} (a,cb) \\
    = (-1)^{\bar a \bar c + \bar b \bar c} (cb,\psi(a))
    = (-1)^{\bar a \bar c + \bar b \bar c} (c,b\psi(a))
    = \varphi(b\psi(a))(c).
  \end{multline*}
  Thus, $\varphi(b) \cdot a = \varphi(b \psi(a))$.
\end{proof}

Suppose $\tau$ is an automorphism of $B$.  Then, if $M$ is a graded left (resp.\ right) $B$-supermodule, it is straightforward to verify that
\begin{equation} \label{eq:tau-duality}
  (^\tau M)^\vee \cong (M^\vee)^\tau \qquad (\text{resp.\ } (M^\tau)^\vee \cong {^\tau(M^\vee)})
\end{equation}
as right (resp.\ left) graded $B$-supermodules.
\details{
  Suppose $M$ is a graded left $B$-supermodule.  We have $(^\tau M)^\vee =\HOM_\F(^\tau M,\F)$, with the right action given by~\eqref{eq:left-module-dual-right-action}. Hence, for $f\in\HOM_\F(^\tau M,\F)$, $b \in B$ and $m \in M$, we have
  \[
    (f\cdot b)(m) = f(b\cdot m)=f(\tau(b)m)=(f\circ\pL{\tau(b)})(m)=(f\tau(b))(m).
  \]
  This proves that $\HOM_\F(^\tau M,\F)\cong \HOM_\F(M,\F)^\tau$, which is what we wanted. The proof for right supermodules is similar.
}
Note also that, if $M$ is a graded left $B$-supermodule and $N$ is a graded right $B$-supermodule, then we have an isomorphism
\begin{equation} \label{eq:twist-across-tensor-prod}
  N^\tau \otimes_B M \cong N \otimes_B {^{\tau^{-1}} M}.
\end{equation}

\begin{lem} \label{lem:rickard}
  For $i=1,2$, let $B_i$ be a Frobenius graded superalgebra of degree $(-\delta_i,\sigma_i)$, with Nakayama automorphism $\psi_i$.  Let $M$ be a finite-dimensional graded $(B_1,B_2)$-superbimodule that is projective as a left $B_1$-supermodule and also as a right $B_2$-supermodule.  Then the functor
  \[
  	M\otimes_{B_2} - \colon B_2\grm \to B_1\grm
  \]
  has a right adjoint functor
  \[
  	(M^{\vee})^{\psi_1^{-1}}\{\delta_1,\sigma_1\} \otimes_{B_1} - \colon B_1\grm\to B_2\grm,
  \]
  and a left adjoint functor
  \[
  	^{\psi_2}(M^{\vee})\{\delta_2,\sigma_2\} \otimes_{B_1} - \colon B_1\grm\to B_2\grm.
  \]
\end{lem}

\begin{proof}
  Recall the left $B_2$-action and right $B_1$-action on $\HOM_{B_1}(M,B_1)$ defined by~\eqref{eq:left-hom-action} and~\eqref{eq:right-hom-action}, respectively.  Suppose $X$ is a graded left $B_1$-supermodule.  Consider the natural $\F$-linear map
  \[
  	\Phi \colon \HOM_{B_1}(M,B_1)\otimes_{B_1} X\to \HOM_{B_1}(M,X),\quad \alpha\otimes x\mapsto \big(~m\mapsto (-1)^{\bar{x}\bar{m}}\alpha(m)x~\big).
  \]
  \details{
    This map is well-defined since we have (see~\eqref{eq:right-hom-action}), for $\alpha \in \HOM_{B_1}(M,B_1)$, $b \in B_1$, $x \in X$,
    \begin{gather*}
      (\alpha \cdot b) \otimes x = \left( (-1)^{\bar b \bar \alpha} (\pr{b}) \circ \alpha \right) \otimes x \mapsto \big(~m \mapsto (-1)^{\bar x \bar m + \bar b \bar m} (\alpha(m)b) x~\big), \qquad \text{and} \\
      \alpha \otimes (b x) \mapsto \big(~m \mapsto (-1)^{(\bar b + \bar x) \bar m} \alpha(m) (b x) = (-1)^{\bar x \bar m + \bar b \bar m} (\alpha(m) b) x~).
    \end{gather*}
  }
  If $M \cong B_1$ as a left $B_1$-supermodule, then $\Phi$ is an isomorphism between the graded vector spaces $\HOM_{B_1}(B_1,B_1)\otimes_{B_1} X$ and $\HOM_{B_1}(B_1,X)$, which are both naturally isomorphic to $X$. If $M$ is free as a left $B_1$-supermodule, then we also get an isomorphism in a similar way. It follows from additivity that $\Phi$ is an isomorphism whenever $M$ is a finite-dimensional projective graded $B_1$-supermodule.
  \details{
    If $P\oplus Q=F$ where $F$ is a free supermodule, then $\Phi$ gives an isomorphism
    \[
      \Phi \colon (\HOM_{B_1}(P,B_1) \otimes_{B_1} X) \oplus (\HOM_{B_1}(Q,B_1) \otimes_{B_1} X) \to \HOM_{B_1}(P,X) \oplus \HOM_{B_1}(Q,X).
    \]
    It is clear from the definition of $\Phi$ that $\Phi(\HOM_{B_1}(P,B_1)\otimes_{B_1} X)\subseteq  \HOM_{B_1}(P,X)$ and $\Phi(\HOM_{B_1}(Q,B_1)\otimes_{B_1} X)\subseteq  \HOM_{B_1}(Q,X)$.  The claim follows.
  }
  The isomorphism $\Phi$ is natural in $M$, hence it is also an isomorphism of graded left $B_2$-supermodules, for the action~\eqref{eq:left-hom-action}, when $M$ satisfies the conditions in the statement of the lemma.  Since $\Phi$ is also natural in $X$, we have an isomorphism of functors
  \[
    \HOM_{B_1}(M,B_1) \otimes_{B_1} - \cong \HOM_{B_1}(M,-)
  \]
  from $B_1\grm$ to $B_2\grm$.

  From Lemma~\ref{lem:dual-module-isom}, we have an isomorphism of graded right $(B_1,B_1)$-superbimodules $B_1 \cong \HOM_\F(B_1,\F)^{\psi_1^{-1}}\{\delta_1,\sigma_1\}$.  We then get isomorphisms of graded $(B_2,B_1)$-superbimodules
  \begin{align*}
    \HOM_{B_1}(M,B_1) &\cong \HOM_{B_1}(M,\HOM_\F(B_1,\F)^{\psi_1^{-1}}\{\delta_1,\sigma_1\}) \\
    &\cong \HOM_{B_1}(M,\HOM_\F(B_1,\F))^{\psi_1^{-1}}\{\delta_1,\sigma_1\} \\
    &\cong \HOM_\F(B_1 \otimes_{B_1} M,\F)^{\psi_1^{-1}}\{\delta_1,\sigma_1\} \\
    &\cong \HOM_\F(M,\F)^{\psi_1^{-1}}\{\delta_1,\sigma_1\} \\
    &\cong (M^\vee)^{\psi_1^{-1}}\{\delta_1,\sigma_1\},
  \end{align*}
  where the third isomorphism follows from the tensor-hom adjunction (it is a straightforward exercise to verify that this is an isomorphism of superbimodules).
  \details{
    The third isomorphism above follows from the fact that we have an isomorphism of $(B_2,B_1)$-superbimodules
    \[
      \HOM_\F(B_1 \otimes_{B_1} M,\F) \cong \HOM_{B_1}(M,\HOM_\F(B_1,\F)).
    \]
    That this is an isomorphism of graded $\F$-vector spaces follows from the tensor-hom adjunction.  The left $B_2$-action on both sides comes from the right action of $B_2$ on $M$, as in~\eqref{eq:left-hom-action}. The right action of $B_1$ on the RHS side comes from \eqref{eq:right-hom-action} because $\HOM_\F(B_1,\F)$ is a right $B_1$-supermodule (and this comes from the left multiplication action of $B_1$ on itself, from \eqref{eq:left-module-dual-right-action}). The right action of $B_1$ on the LHS comes from the left action on $B_1\otimes_{B_1}M$ as in~\eqref{eq:left-module-dual-right-action}.  Now, the isomorphism respects these actions because it is defined as follows: if $f\in\HOM_\F(B_1 \otimes_{B_1} M,\F)$, we define $\varphi_f\in  \HOM_{B_1}(M,\HOM_\F(B_1,\F))$ by defining, for all $m\in M$, a map $\varphi_f(m) \in \HOM_\F(B_1,\F)$. For all $b\in B_1$ we have
    \[
      \varphi_f(m)(b)=(-1)^{\bar b \bar m}f(b\otimes m).
    \]
    Now, for the left $B_2$-action, we have
    \begin{multline*}
      \varphi_{b_2f}(m)(b) = (-1)^{\bar b \bar m}(b_2f)(b\otimes m) = (-1)^{\bar b \bar m + \bar b_2 \bar f + \bar b_2 \bar b + \bar b_2 \bar m} f(b\otimes mb_2) \\
      = (-1)^{\bar b_2 \bar f + \bar b_2 \bar m} \varphi_f(mb_2)(b)
      = (-1)^{\bar b_2 \bar f} (\varphi_f \circ \pr{b_2})(m)(b) = (b_2\varphi_f)(m)(b).
    \end{multline*}
    So $\varphi$ is indeed a map of left $B_2$-supermodules.  For the right action of $B_1$, we have
    \begin{multline*}
      \varphi_{fb_1}(m)(b) = (-1)^{\bar b \bar m}(fb_1)(b\otimes m) = (-1)^{\bar b \bar m} f(b_1b\otimes m) = (-1)^{\bar b_1 \bar m} \varphi_f(m)\circ\pL{b_1}(b) \\
      = (-1)^{\bar b_1 \bar m} (\varphi_f(m) b_1)(b) = (-1)^{\bar b_1 \bar f} (\pr{b_1}\circ\varphi_f)(m)(b) = (\varphi_fb_1)(m)(b).
    \end{multline*}
  }
  It follows that we have isomorphisms of functors
  \[
    (M^\vee)^{\psi_1^{-1}}\{\delta_1,\sigma_1\} \otimes_{B_1} - \cong \HOM_{B_1}(M,B_1) \otimes_{B_1} - \cong \HOM_{B_1}(M,-).
  \]
  Since the functor $M\otimes_{B_2} -$ has the right adjoint $\HOM_{B_1}(M,-)$, the first statement of the lemma holds.

  For the second statement, remark that if $M$ is finite dimensional and projective as a left $B_1$-supermodule and as a right $B_2$-supermodule, then $M^\vee$ is also finite dimensional and projective as a left $B_2$-supermodule and right $B_1$-supermodule, and thus so is $^{\psi_2}(M^\vee)$.
  \details{It is clear for free supermodules.  For projective supermodules, we use the fact that $(M\oplus N)^\vee\cong M^\vee\oplus N^\vee$.}
  Thus we can apply the first part of the proof to conclude that
  $^{\psi_2}M^\vee\{\delta_2,\sigma_2\} \otimes_{B_1}-$ is left adjoint to
  \[
    ((^{\psi_2}M^\vee\{\delta_2,\sigma_2\})^\vee)^{\psi_2^{-1}}\{\delta_2,\sigma_2\} \otimes_{B_2}- \cong ((^{\psi_2}M^\vee)^\vee)^{\psi_2^{-1}} \otimes_{B_2}-.
  \]
  But, using \eqref{eq:tau-duality}, we have isomorphisms of $(B_1,B_2)$-superbimodules
  \[
    ((^{\psi_2}M^\vee)^\vee)^{\psi_2^{-1}} \cong (((M^\vee)^\vee)^{\psi_2})^{\psi_2^{-1}} \cong (M^\vee)^\vee \cong M,
  \]
  which concludes the proof.
\end{proof}

\begin{prop} \label{prop:Frobenius-twisted}
  If $A$ is a tower of algebras such that each $A_\lambda$, $\lambda \in \Lambda$, is a Frobenius graded superalgebra of degree $(-\delta_\lambda, \sigma_\lambda)$ with Nakayama automorphism $\psi_\lambda$, then induction is conjugate shifted right adjoint to restriction with conjugation $\Psi$ and shifting $\{\delta,\sigma\}$ (in the notation of Definition~\ref{def:twisted-adjoint}).
\end{prop}

\begin{proof}
  We apply Lemma~\ref{lem:rickard} with $B_1 = A_{\lambda+\mu}$, $B_2 = A_\lambda \otimes A_\mu$, and $M = A_{\lambda+\mu}$, considered as an $(A_{\lambda+\mu},A_\lambda \otimes A_\mu)$-superbimodule in the natural way.  In the notation of Lemma~\ref{lem:rickard}, we have
  \[
    \psi_2 = \psi_\lambda \otimes \psi_\mu,\quad \delta_2 = \delta_\lambda + \delta_\mu,\quad \sigma_2 = \sigma_\lambda + \sigma_\mu.
  \]
  By Lemma~\ref{lem:dual-module-isom}, we have
  \[
    (A_{\lambda+\mu})^\vee \cong (A_{\lambda+\mu})^{\psi_{\lambda+\mu}}\{-\delta_{\lambda+\mu}, -\sigma_{\lambda+\mu}\},
  \]
  as $(A_{\lambda+\mu},A_{\lambda+\mu})$-superbimodules.  Restricting the action on the left, we have
  \[
    M^\vee \cong {_{A_\lambda \otimes A_\mu}(A_{\lambda+\mu})}^{\psi_{\lambda+\mu}}\{-\delta_{\lambda+\mu}, -\sigma_{\lambda+\mu}\},
  \]
  as $(A_\lambda \otimes A_\mu, A_{\lambda+\mu})$-superbimodules.  Then, by Lemma~\ref{lem:rickard}, the left adjoint to $\Ind_{A_\lambda \otimes A_\mu}^{A_{\lambda+\mu}} = M \otimes_{B_2} -$ is
  \begin{multline*}
    ^{\psi_\lambda \otimes \psi_\mu} ({_{A_\lambda \otimes A_\mu}(A_{\lambda+\mu})})^{\psi_{\lambda+\mu}}\{ \delta_\lambda + \delta_\mu-\delta_{\lambda+\mu},   \sigma_\lambda+ \sigma_\mu-\sigma_{\lambda+\mu}\} \otimes_{A_{\lambda+\mu}} - \\ = (\Psi_\lambda \otimes \Psi_\mu)\Res^{A_{\lambda+\mu}}_{A_\lambda \otimes A_\mu} \Psi_{\lambda + \mu}^{-1}\{ \delta_\lambda + \delta_\mu-\delta_{\lambda+\mu},   \sigma_\lambda+ \sigma_\mu-\sigma_{\lambda+\mu}\},
  \end{multline*}
  where the exponent of $-1$ on $\Psi_{\lambda+\mu}$ comes from~\eqref{eq:twist-across-tensor-prod}.
\end{proof}

%%%%%%%%%%%%%%%%%%%%%%%%%%%%%%%%%%%%%%%%%%%%%%%%%%%%
%
\section{Towers of wreath product algebras} \label{sec:wreath}
%
%%%%%%%%%%%%%%%%%%%%%%%%%%%%%%%%%%%%%%%%%%%%%%%%%%%%

In this section, we introduce a large class of examples of strong compatible towers.  This class includes several towers that have been considered in the literature (see Example~\ref{eg:wreath-examples}).

Suppose that $B$ is a graded superalgebra over an algebraically closed field $\F$.  Recall that $B^{\otimes n}$ is a graded superalgebra with multiplication
\[
  (b_1 \otimes \dotsb \otimes b_n)(b_1' \otimes \dotsb \otimes b_n') = (-1)^{\sum_{i<j} \bar b_i \bar b_j'} b_1 b_1' \otimes \dotsb \otimes b_n b_n',\quad b_i,b_i' \in B,\ 1 \le i \le n.
\]
Then $S_n$ acts on $B^{\otimes n}$ by superpermutations.  More precisely, if $s_k \in S_n$ is the simple transposition $(k,k+1)$ for $1 \le k \le n-1$, then
\[
  s_k \cdot (b_1 \otimes \dotsb \otimes b_n)
  = (-1)^{\bar b_k \bar b_{k+1}} b_1 \otimes \dotsb \otimes b_{k-1} \otimes b_{k+1} \otimes b_k \otimes b_{k+2} \otimes \dotsb \otimes b_n.
\]
(Note that we consider the action by superpermutations since the action by usual permutations does not, in general, give an action of $S_n$ on $B^{\otimes n}$ by superalgebra automorphisms.)  We can thus form the algebra $A_n := B^{\otimes n} \rtimes S_n$.  As an $\F$-vector space, we have $B^{\otimes n} \rtimes S_n \cong B^{\otimes n} \otimes \F[S_n]$.  This inherits a $(\Z \times \Z_2)$-grading from $B$ by declaring $S_n$ to be in degree zero.  The multiplication is determined by the fact that $B^{\otimes n}$ and $\kk[S_n]$ are sub-superalgebras and
\[
  \tau \beta \tau^{-1} = \tau \cdot \beta,\quad \beta \in B^{\otimes n},\ \tau \in S_n.
\]
Motivated by the analogous construction for groups, we call $B^{\otimes n} \rtimes S_n$ a \emph{wreath product algebra}.

\begin{eg} \label{eg:wreath-examples}
  If $B$ is the rank one Clifford superalgebra, then $B^{\otimes n} \rtimes S_n$ is isomorphic to the Sergeev superalgebra (see~\cite[Lem.~13.2.3]{Kle05}).  Thus, the tower of Sergeev superalgebras fits into the framework of the current section.  The algebras considered in~\cite{CL12} are also of the form presented here.
\end{eg}

We define an external multiplication on $A = \bigoplus_{n \in \N} A_n$ by
\begin{multline} \label{def-eq:ext-mult}
  \rho_{m,n} \colon A_m \otimes A_n \cong (B^{\otimes m} \rtimes S_m) \otimes (B^{\otimes n} \rtimes S_n) \\
  \cong B^{\otimes (m+n)} \rtimes (S_m \times S_n) \hookrightarrow B^{\otimes (m+n)} \rtimes S_{m+n} \cong A_{m+n},
\end{multline}
induced by the natural inclusion $S_m \times S_n \hookrightarrow S_{m+n}$.  Axioms~\ref{item:TA1} and~\ref{item:TA2} follow immediately.

Recall that, as a left $(\F[S_m] \otimes \F[S_n])$-module, $\F[S_{m+n}]$ has a basis given by minimal length representatives of the cosets $(S_m \times S_n) \backslash S_{m+n}$.  It follows that, as a left $(A_m \otimes A_n)$-supermodule, we have
\[ \ts
  A_{m+n} = \bigoplus_{w \in X_{m,n}} (A_m \otimes A_n) w,
\]
where $X_{m,n}$ is a set of minimal length representatives of the cosets $(S_m \times S_n) \backslash S_{m+n}$.  Thus $A_{m+n}$ is a projective left $(A_m \otimes A_n)$-supermodule.  Similarly, it is also a projective right $(A_m \otimes A_n)$-supermodule and so axiom~\ref{item:TA3} is satisfied.   Finally,~\ref{item:TA4} is satisfied since $\F$ is algebraically closed.

Now suppose that $B$ is a Frobenius graded superalgebra of degree $(-\delta,\sigma)$.  If $\tr_{B} \colon B \to \F\{\delta,\sigma\}$ is the trace map of $B$, then, by Lemma~\ref{lem:Frobenius-alg-tensor-prod}, $B^{\otimes n}$ is a Frobenius graded superalgebra with trace map $\tr_B^{\otimes n}$.  Recall also that $\F[S_n]$  is a Frobenius graded superalgebra (concentrated in degree zero) with trace map $\tr_{S_n} \colon \F[S_n] \to \F$ given by $\tr_{S_n}(w) = \delta_{w,w_0}$, where $w_0$ is the longest element of $S_n$.

\begin{lem} \label{lem:wreath-Frobenius}
  The algebra $A_n$ is a Frobenius graded superalgebra of degree $(-n\delta,n\sigma)$, with trace map $\tr_n \colon A_n \to \F$ given by $\tr_n = \tr_B^{\otimes n} \otimes \tr_{S_n}$.  The corresponding Nakayama automorphism $\psi_n \colon A_n \to A_n$ is given by
  \begin{gather*} \ts
    \psi_n(b_1 \otimes \dotsb \otimes b_n) = (-1)^{\sum_{i<j}\bar b_i \bar b_j} \psi_B(b_n) \otimes \dotsb \otimes \psi_B(b_1),\quad b_1,\dotsc,b_n \in B,\\
    \psi_n(s_i) = (-1)^\sigma s_{n-i},\quad i=1,\dotsc,n-1,
  \end{gather*}
  where $\psi_B \colon B \to B$ is the Nakayama automorphism of $B$.
\end{lem}

\begin{proof}
  To show that $A_n$ is a Frobenius graded superalgebra with trace map $\tr_n$, it suffices to show that $\ker \tr_n$ contains no nonzero left ideals.  Let $I$ be a nonzero left ideal of $A_n$.  Then choose a nonzero element $a = \sum_{w \in S_n} a_w w \in I$.  Without loss of generality, we may assume that $a_{w_0} \ne 0$ (otherwise consider $w_0 \tau^{-1} a$, where $\tau$ is a maximal length element of the set $\{w \in S_n\ |\ a_w \ne 0\}$).  Now, since $B^{\otimes n}$ is a Frobenius algebra, the ideal of $B^{\otimes n}$ generated by $a_{w_0}$ is not contained in $\ker \tr_B^{\otimes n}$.  Thus, there exists $c \in B^{\otimes n}$ such that $\tr_B^{\otimes n}(c a_{w_0}) \ne 0$.  Then $\tr_n(ca) \ne 0$.  Thus, $I$ is not contained in $\tr_n$.  It is clear that the degree of the Frobenius graded superalgebra $A_n$ is $(-n\delta, n\sigma)$.

  Now suppose $\beta \in B^{\otimes n}$ and $\tau \in S_n$ such that $\tau s_i = w_0$ (hence $s_{n-i} \tau = w_0$).  Write $\beta = \beta' + \beta''$, where
  \[
    \beta' \in (\bar B_\sigma)^{\otimes n},\quad \beta'' \in \bigoplus_{\substack{\epsilon_1,\dotsc,\epsilon_n \\ \text{not all equal to } \sigma}} \left( \bar B_{\epsilon_1} \otimes \dotsb \otimes \bar B_{\epsilon_n} \right).
  \]
  Then
  \begin{multline*}
    \tr_n((\beta \otimes \tau)s_i) = \tr_n(\beta \otimes w_0) = \tr_B^{\otimes n}(\beta') = (-1)^\sigma \tr_B^{\otimes n}(s_{n-i} \cdot \beta') \\
    = (-1)^\sigma \tr_B^{\otimes n}(s_{n-i} \cdot \beta) = (-1)^\sigma \tr_n(s_{n-i} (\beta \otimes \tau)).
  \end{multline*}
  On the other hand, if $\tau s_i \ne w_0$ (in which case $s_{n-i} \tau \ne w_0$), then
  \[
    \tr_n((\beta \otimes \tau)s_i) = 0 = \tr_n(s_{n-i}(\beta \otimes \tau)).
  \]
  Thus, $\psi_n(s_i)= (-1)^\sigma s_{n-i}$.

  Finally, for homogeneous $b_1,\dotsc,b_n \in B$ and $\beta \in B^{\otimes n}$, we have
  \begin{align*}
    \tr_n((b_1 \otimes \dotsb \otimes b_n)(\beta \otimes w_0)) &= \tr((b_1 \otimes \dotsb \otimes b_n)\beta \otimes w_0) \\
    &= \ts \tr_B^{\otimes n}((b_1 \otimes \dotsb \otimes b_n)\beta) \\
    &= \ts (-1)^{|\beta|(|b_1| + \dotsb + |b_n|)} \tr_B^{\otimes n}(\beta(\psi_B(b_1) \otimes \dotsb \otimes \psi_B(b_n))) \\
    &= \ts (-1)^{|\beta|(|b_1| + \dotsb + |b_n|)} (-1)^{\sum_{i<j}\bar b_i \bar b_j} \tr_n ((\beta \otimes w_0)(\psi_B(b_n) \otimes \dotsb \otimes \psi_B(b_1))).
  \end{align*}
  On the other hand, if $\tau \ne w_0$, then we have
  \[ \ts
    \tr_n((b_1 \otimes \dotsb \otimes b_n)(\beta \otimes \tau)) = 0 = \tr_n ((\beta \otimes \tau)(\psi_B(b_n) \otimes \dotsb \otimes \psi_B(b_1))).
  \]
  This completes the proof.
\end{proof}

\begin{prop} \label{prop:wreath-tower-strong}
  The tower $A$ is strong with trivial twist and conjugation $\Psi$ given by the Nakayama automorphism (see Proposition~\ref{prop:Frobenius-twisted}).
\end{prop}

\begin{proof}
  It follows from Lemma~\ref{lem:wreath-Frobenius} and Proposition~\ref{prop:Frobenius-twisted} that induction is conjugate shifted right adjoint to restriction.  Since $\alpha_{n+m} - \alpha_n - \alpha_n = (n+m)\alpha - n\alpha - m\alpha = 0$ for $\alpha = \delta, \sigma$, condition~\ref{item:S1} of Definition~\ref{def:strong} is satisfied with $d=\epsilon=0$ and $\kappa=0$.

  To check~\ref{item:S2}, we formulate the isomorphism~\eqref{eq:strong-isom} in terms of superbimodules.  Fix $n,m,k,\ell$ such that $n+m=k+\ell$ and set $K=n+m$. Let $_{(k,\ell)}(A_K)_{(n,m)}$ denote $A_K$, thought of as a $(A_k\otimes A_\ell,A_n\otimes A_m)$-superbimodule in the natural way.  Then we have
  \[
    \Res^{A_K}_{A_k \otimes A_\ell} \Ind^{A_K}_{A_n \otimes A_m} \cong {_{(k,\ell)}(A_K)_{(n,m)}} \otimes -.
  \]
  On the other hand, for each $r \in \N$ satisfying $k-m=n-\ell \le r \le \min \{n,k\}$, we have
  \begin{gather*}
    \Ind^{A_k \otimes A_\ell}_{A_r \otimes A_{n-r} \otimes A_{k-r} \otimes A_{\ell+r-n}} S_{23} \Res^{A_n \otimes A_m}_{A_r \otimes A_{n-r} \otimes A_{k-r} \otimes A_{\ell+r-n}} \cong M_r \otimes -,\quad \text{ where} \\
    M_r=(A_k\otimes A_\ell) \otimes_{A_r \otimes A_{k-r} \otimes A_{n-r} \otimes A_{\ell+r-n}} S^{23} \otimes_{A_r\otimes A_{n-r} \otimes A_{k-r}\otimes A_{\ell+r-n}} (A_n\otimes A_m),
  \end{gather*}
  and where $S^{23} = A_r\otimes A_{n-r} \otimes A_{k-r}\otimes A_{\ell+r-n}$ (as an $\F$-module) is a right $(A_r \otimes A_{n-r} \otimes A_{k-r} \otimes A_{\ell + r - n})$-supermodule in the natural way and is a left $(A_r \otimes A_{k-r} \otimes A_{n-r} \otimes A_{\ell+r-n})$-supermodule via the map
  \begin{gather*}
    A_r \otimes A_{k-r} \otimes A_{n-r} \otimes A_{\ell+r-n} \to A_r\otimes A_{n-r} \otimes A_{k-r}\otimes A_{\ell+r-n},\\
    a_1 \otimes a_2 \otimes a_3 \otimes a_4 \mapsto (-1)^{\bar a_2 \bar a_3} a_1 \otimes a_3 \otimes a_2 \otimes a_4.
  \end{gather*}
  Therefore, in order to prove the isomorphism~\eqref{eq:strong-isom}, it suffices to prove that we have an isomorphism of superbimodules
  \begin{equation} \label{eq:wreath-S2-isom} \ts
    _{(k,\ell)}(A_K)_{(n,m)} \cong \bigoplus_{r=n-\ell}^{\min\{n,k\}} M_r.
  \end{equation}

  For each $r \in \N$ satisfying $k-m = n - \ell \le r \le \min \{n,k\}$, define $w_r \in S_K$ by
  \[
    w_r(i) =
    \begin{cases}
      i & \text{if } 1 \le i \le r, \\
      i - r + k & \text{if } r < i \le n, \\
      i - n + r & \text{if } n < i \le n+k-r, \\
      i & \text{if } n+k-r < i \le K,
    \end{cases}
  \]
  Then the $w_r$ form a complete set of minimal length representatives of the double cosets in $S_k\times S_\ell \setminus S_K / S_n \times S_m$.  If $C_w$ is the double coset containing $w_r$, then its cardinality is
  \[
    |C_r| = m!n! \binom{k}{r} \binom{\ell}{n-r}.
  \]
  (See~\cite[proof of Prop.~4.3]{SY13} for further details.)  One easily verifies that
  \begin{equation} \label{eq:wr-commutation}
    w_r (a_1 \otimes a_2 \otimes a_3 \otimes a_4) = (-1)^{\bar a_2 \bar a_3} (a_1 \otimes a_3 \otimes a_2 \otimes a_4) w_r
  \end{equation}
  for all $a_1 \in A_r$, $a_2 \in A_{n-r}$, $a_3 \in A_{k-r}$, $a_4 \in A_{\ell+r-n}$.

  Define $M_r' \subseteq {_{(k,\ell)}(A_K)_{(n,m)}}$ to be the sub-superbimodule generated by $B^{\otimes r} \otimes w_r$.  Then
  \[ \ts
    {_{(k,\ell)}(A_K)_{(n,m)}}=\bigoplus_{r=n-\ell}^{\min\{n,k\}} M_r'
  \]
  and $\dim_\F M_r' = |C_r|(r \dim_\F B)$.  Note that the dimension of $A_k \otimes A_\ell$ as a right supermodule over $A_r \otimes A_{n-r} \otimes A_{k-r} \otimes A_{\ell-n+r}$ is $k! \ell! / r! (n-r)! (k-r)! (\ell - n + r)!$ and that the dimension of $A_m \otimes A_n$ as a left supermodule over $A_r \otimes A_{n-r} \otimes A_{k-r} \otimes A_{\ell+r-n}$ is $m! n! / r! (n-r)! (k-r)! (\ell - n+ r)!$.  Therefore, $\dim_\F M_r = |C_r|(r \dim_\F B) = \dim_\F M_r'$.  Now consider the $(A_k \otimes A_\ell, A_n \otimes A_m)$-superbimodule map $M_r \to M_r'$ uniquely determined by
  \begin{equation} \label{eq:wreath-strong-tower-bimodule-map}
    1_{A_k \otimes A_\ell} \otimes 1_{A_n \otimes A_m} \mapsto w_r,
  \end{equation}
  which is well defined by~\eqref{eq:wr-commutation}.  This map is surjective, and thus is an isomorphism by dimension considerations.
\end{proof}

\begin{lem} \label{lem:Hecke-like-coprod-twisting}
  We have an isomorphism of functors $\Psi^{\otimes 2} \Delta \Psi^{-1} \cong S_{12} \Delta$ on $A\md$ (hence also on $A\pmd$).
\end{lem}

\begin{proof}
  It suffices to prove that, for $m, n \in \N$, we have an isomorphism of functors
  \begin{equation} \label{eq:Hecke-like-coprod-twisting}
    (\Psi_m \otimes \Psi_n) \circ \Res^{A_{m+n}}_{A_m \otimes A_n} \circ \Psi_{m+n}^{-1} \cong S_{12} \circ \Res^{A_{m+n}}_{A_n \otimes A_m}.
  \end{equation}
  Note that, for $k \in \N$, $\Psi_k^{-1} = A_k^{\psi} \otimes_{A_k} -$, where $A_k^\psi$ denotes $A_k$, considered as an $(A_k,A_k)$-superbimodule with the right action twisted by $\psi_k$.  Similarly, $\Psi_k = A_k^{\psi^{-1}} \otimes_{A_k} - $.  Thus, describing each functor in~\eqref{eq:Hecke-like-coprod-twisting} as tensoring on the left by the appropriate superbimodule, it suffices to prove that we have an isomorphism of superbimodules
  \begin{equation} \label{eq:bimodule-isom-comult}
    (A_m^{\psi^{-1}} \otimes A_n^{\psi^{-1}}) \otimes_{A_m \otimes A_n} A_{m+n}^\psi \cong S \otimes_{A_n \otimes A_m} A_{m+n},
  \end{equation}
  where $S$ is $A_n \otimes A_m$ considered as an $(A_m \otimes A_n, A_n \otimes A_m)$-supermodule via the obvious right action and with left action given by $(a_1 \otimes a_2, s) \mapsto (-1)^{\bar a_1 \bar a_2}(a_2 \otimes a_1) s$ for $s \in S$ and homogeneous $a_1 \in A_m$, $a_2 \in A_n$.

  For $k \in \N$, let $1_k$ denote the identity element of $A_k$ and $1_k^\psi$ denote this same element considered as an element of $A_k^\psi$ (and similarly, with $\psi$ replaced by $\psi^{-1}$).   It is straightforward to show that the map between the superbimodules in~\eqref{eq:bimodule-isom-comult} given by
  \[
    (1_m^{\psi^{-1}} \otimes 1_n^{\psi^{-1}}) \otimes 1_{m+n}^\psi \mapsto (1_n \otimes 1_m) \otimes 1_{m+n}.
  \]
  (and extended by linearity) is a well-defined isomorphism.
  \details{
  For $i \in \{1,\dotsc,m\}$, $j \in \{1,\dotsc,n\}$, and homogeneous $b \in B^{\otimes n}$, $b' \in B^{\otimes m}$, in $S \otimes_{A_n \otimes A_m} A_{m+n}$ we have
  \begin{gather*}
    (s_i \otimes 1) \left( (1_n \otimes 1_m) \otimes 1_{m+n} \right) = (1_n \otimes s_i) \otimes 1_{m+n} = \left( (1_n \otimes 1_m) \otimes 1_{m+n} \right) s_{n+i}, \\
    (1 \otimes s_j) \left( (1_n \otimes 1_m) \otimes 1_{m+n} \right) = (s_j \otimes 1_m) \otimes 1_{m+n} = \left( (1_n \otimes 1_m) \otimes 1_{m+n} \right) s_j, \\
    (b \otimes b') \left( (1_n \otimes 1_m) \otimes 1_{m+n} \right) = (-1)^{\bar b \bar b'}(b' \otimes b) \otimes 1_{m+n} = (-1)^{\bar b \bar b'}\left( (1_n \otimes 1_m) \otimes 1_{m+n} \right) (b' \otimes b),
  \end{gather*}
  and in $(A_m^{\psi^{-1}} \otimes A_n^{\psi^{-1}}) \otimes_{A_m \otimes A_n} A_{m+n}^\psi$ we have
  \begin{gather*}
    (s_i \otimes 1) \left( (1_m^{\psi^{-1}} \otimes 1_n^{\psi^{-1}}) \otimes 1^\psi_{m+n} \right) = (-1)^\sigma \left( (1_m^{\psi^{-1}} \otimes 1_n^{\psi^{-1}}) (s_{m-i} \otimes 1) \right) \otimes 1_{m+n}^\psi = \left( (1_m^{\psi^{-1}} \otimes 1_n^{\psi^{-1}}) \otimes 1_{m+n}^\psi \right) s_{n+i}, \\
    (1 \otimes s_j) \left( (1_m^{\psi^{-1}} \otimes 1_n^{\psi^{-1}}) \otimes 1^\psi_{m+n} \right) = (-1)^\sigma \left( (1_m^{\psi^{-1}} \otimes 1_n^{\psi^{-1}}) (1 \otimes s_{n-j}) \right) \otimes 1_{m+n}^\psi = \left( (1_m^{\psi^{-1}} \otimes 1_n^{\psi^{-1}}) \otimes 1_{m+n}^\psi \right) s_j
  \end{gather*}
  and
  \begin{multline*}
    (b \otimes b') \left( (1_m^{\psi^{-1}} \otimes 1_n^{\psi^{-1}}) \otimes 1_{m+n}^\psi \right) = \left( (1_m^{\psi^{-1}} \otimes 1_n^{\psi^{-1}})(\psi_n(b) \otimes \psi_m(b')) \right) \otimes 1_{m+n}^\psi \\
    = (-1)^{\bar b \bar b'} \left( (1_m^{\psi^{-1}} \otimes 1_n^{\psi^{-1}}) \otimes 1_{m+n}^\psi \right) (b' \otimes b).
  \end{multline*}
  Thus we have a well defined map between the superbimodules in~\eqref{eq:bimodule-isom-comult} given by
  \[
    (1_m^{\psi^{-1}} \otimes 1_n^{\psi^{-1}}) \otimes 1_{m+n}^\psi \mapsto (1_n \otimes 1_m) \otimes 1_{m+n}.
  \]
  This map is clearly surjective.  Since we can construct the inverse map analogously, it is an isomorphism.
  }
\end{proof}

\begin{lem} \label{lem:wreath-dualizing}
  We have isomorphisms of functors $\nabla \cong \nabla S_{12}$ and $\Delta \cong S_{12} \Delta$ on $A\md$ (hence also on $A\pmd$).
\end{lem}

\begin{proof}
  Let $m, n \in \N$ and define $w \in S_{m+n}$ by
  \[
    w(i) =
    \begin{cases}
      i+m & \text{if } 1 \le i \le n, \\
      i-n & \text{if } n < i \le m+n.
    \end{cases}
  \]
  Then we have $w s_i = s_{w(i)} w$ for all $i=1,\dotsc,m-1,m+1,\dotsc,m+n-1$.  Now, let $S$ be $A_m \otimes A_n$ considered as an $(A_n \otimes A_m, A_m \otimes A_n)$-supermodule via the obvious right action and with left action given by $(a_1 \otimes a_2, s) \mapsto (-1)^{\bar a_1 \bar a_2} (a_2 \otimes a_1) s$, for $s \in S$ and homogeneous $a_1 \in A_n$, $a_2 \in A_m$.  So we have an isomorphism of functors $S_{12} \cong S \otimes -$.  It is straightforward to verify that the map
  \[
    A_{m+n} \to A_{m+n} \otimes_{A_n \otimes A_m} S,\quad a \mapsto a w \otimes (1 \otimes 1),
  \]
  is an isomorphism of $(A_{m+m}, A_m \otimes A_n)$-superbimodules.  It follows that $\nabla \cong \nabla S_{12}$.  The proof that $\Delta \cong S_{12} \Delta$ is analogous.
\end{proof}

\begin{cor} \label{cor:wreath-strong-compatible}
  Towers of wreath product algebras are strong and compatible.
\end{cor}

\begin{proof}
  It follows immediately from Proposition~\ref{prop:wreath-tower-strong}, Proposition~\ref{prop:dualizing-twisting-isom}, Lemma~\ref{lem:Hecke-like-coprod-twisting}, and Lemma~\ref{lem:wreath-dualizing} that towers of wreath product algebras are strong and dualizing.  Since we saw in the proof of Proposition~\ref{prop:wreath-tower-strong} that $\chi'=\gamma'=0$, they are also compatible.
\end{proof}

\begin{rem}
  The towers of wreath product algebras described in this section should yield categorifications of (quantum) lattice Heisenberg algebras.  For instance, this can be seen in~\cite{CL12} for a particular choice of $B$.  For a description of the (quantum) lattice Heisenberg algebras as Heisenberg doubles, see~\cite[\S\S 7, 8]{RS14}.  It would be interesting to work out the details of this categorification for general $B$.
\end{rem}

%%%%%%%%%%%%%%%%%%%%%%%%%%%%%
%
\section{The tower of nilCoxeter graded superalgebras} \label{sec:Weyl}
%
%%%%%%%%%%%%%%%%%%%%%%%%%%%%%

In this section we specialize the constructions of Section~\ref{sec:categorification} to the tower of nilCoxeter graded superalgebras of type $A$.  We will see that we recover a categorification of the polynomial representation of the quantum Weyl algebra.  Such a categorification was sketched in~\cite[\S4.1]{Kho01} (see Remark~\ref{rem:Khovanov-quantum-Weyl}).  Recall that $\F$ is an arbitrary field of characteristic not equal to two.

\begin{defin}[NilCoxeter graded superalgebra]\label{def:nilcox-super}
  Define the \emph{nilCoxeter graded superalgebra} with parameters $(d,\epsilon)\in\Z\times\Z_2$, denoted by $\Ner_n$, to be the graded unital $\F$-superalgebra with generators $u_1,\dotsc,u_{n-1}$, such that $(|u_i|, \bar u_i)=(d,\epsilon)$, for $i=1,\dotsc,n-1$, subject to the relations
  \begin{gather*}
    u_i^2 = 0\quad \text{for } i=1,2,\dotsc,n-1, \\
    u_iu_j =(-1)^\epsilon u_ju_i\quad \text{for } i,j = 1,\dotsc,n-1 \text{ such that } |i-j|>1, \\
    u_iu_{i+1}u_i = u_{i+1}u_iu_{i+1}\quad \text{for } i=1,2,\dotsc,n-2.
  \end{gather*}
\end{defin}

Since the ideal generated by all the $u_i$ is nilpotent, $\Ner_n$ has, up to grading shift and isomorphism, one simple supermodule, which is one dimensional, and on which the generators $u_i$ all act by zero.   We denote this simple supermodule, in degree zero, by $L_n$.  Since it is one dimensional, it is of type $\mathsf{M}$.  By \cite[Prop.~12.2.12]{Kle05}, the projective cover of $L_n$ is $P_n := \Ner_n$, considered as an $\Ner_n$-supermodule by left multiplication.

The reason we need the factor of $(-1)^\epsilon$ in the relations of Definition \ref{def:nilcox-super} is to satisfy axiom~\ref{item:TA2}.  Indeed, we have a map of graded superalgebras
\begin{gather*}
  \rho_{n,m} \colon \Ner_n\otimes \Ner_m\to \Ner_{n+m}, \\
  u_i\otimes 1\mapsto u_i,\quad 1\otimes u_j \mapsto u_{n+j},\quad 1 \le i \le n-1,\ 1 \le j \le m-1.
\end{gather*}
\details{
  If $1 \le i \le n-1$ and $\ 1 \le j \le m-1$, we have
  \[
    u_i u_{n+j} = \rho_{n,m}((1\otimes u_i)(u_j\otimes 1))=\rho_{n,m} \left( (-1)^{\epsilon^2}u_j\otimes u_i \right) = \rho_{n,m} \left( (-1)^\epsilon (u_j \otimes 1) (1 \otimes u_i) \right) = (-1)^\epsilon u_{n+j} u_i.
  \]
}
Axiom~\ref{item:TA1} is clearly satisfied, and~\ref{item:TA3} holds since $\Ner_{n+m}$ is a free left and right supermodule over $\Ner_{n}\otimes \Ner_m$ of rank ${n+m \choose n}$.  Since, for all $n$, $\dim \Hom_{\Ner_n}(P_n,L_n)=1$ and $L_n$ is of type $\mathsf{M}$, \ref{item:TA4} is also satisfied and $\kk = \Z_{q,\pi}$.

For each $w\in S_n$, and two reduced expressions $w=s_{i_1}\cdots s_{i_r}=s_{j_1}\cdots s_{j_r}$, we have,  by the relations in Definition \ref{def:nilcox-super} that $u_{i_1}\cdots u_{i_r}=\pm u_{j_1}\cdots u_{j_r}$. We can then choose a basis $\{u_w~|~w\in S_n\}$ of $\Ner_n$ that satisfies the following property:
\begin{equation}
  u_v u_w =
  \begin{cases}
    \alpha(v,w) u_{vw} & \text{ if } \ell(v)+\ell(w)=\ell(v+w), \\
    0 & \text{ otherwise},
  \end{cases}
\end{equation}
where $\ell(w)$ is the length of an element $w \in S_n$, and $\alpha \colon S_n\times S_n \to \{\pm 1\} \subseteq \F^\times$ is a nontrivial $2$-cocycle when $\epsilon=1$ and is the trivial cocycle when $\epsilon=0$.
\details{
  When $\epsilon=0$, this is just the usual construction of a basis of the nilCoxeter algebra.  When $\epsilon \neq 0$, it is clear from the relations that we can choose the coefficients to be $\pm 1$.  The function $\alpha$ must be a $2$-cocycle in order for the multiplication to be associative.  Indeed, if the lengths all add up, we have
  \begin{align*} (u_vu_w)u_t &= u_v(u_wu_t) \\
    \iff \alpha(v,w)u_{vw}u_t &= u_v\alpha(w,t)u_{wt} \\
    \iff \alpha(v,w)\alpha(vw,t)u_{vwt} &= \alpha(w,t)\alpha(v,wt)u_{vwt}.
  \end{align*}
}

\begin{lem} \label{lem:grad-nilcox-frob}
  We have that $\Ner_n$ is a Frobenius graded superalgebra of degree $\left(-d{ n \choose 2},\epsilon{n\choose 2}\right)$.
  The Nakayama automorphism is given on the generators by $\psi_n(u_i)=u_{n-i}$.
\end{lem}

\begin{proof}
  Define a trace map $\tr_n \colon \Ner_n\to \F$ by
  \begin{equation}
    \tr_n(u_w) =
    \begin{cases}
      1 & \text{ if } w=w_0, \\
      0 & \text{ if  }w\neq w_0.
    \end{cases}
  \end{equation}
  We need to show that $\ker\tr_n$ does not contain any nonzero left ideal.  Let $I$ be such an ideal and $0\neq a=\sum_{w\in S_n}a_w u_w\in I$.  Choose an element $w'$ of maximal length such that $a_{w'} \neq 0$. Then $\tr_n(u_{w_0(w')^{-1}}a)=\tr_n(\pm a_{w'}u_{w_0})=\pm a_{w'}\neq 0$, which contradicts the hypothesis. The degree of the algebra is simply the negative of the degree of $u_{w_0}$, which is $\left(d{ n \choose 2},\epsilon{n\choose 2}\right)$ because the length of $w_0$ is ${n\choose 2}$.

  Now, for all $w\in S_n$ such that $s_iw\neq w_0$, we have that  $\tr_n(u_iu_w) = 0 = \tr_n(u_wu_{n-i})$.  In the case where $s_iw=w_0=ws_{n-i}$, we have $u_iu_w = (-1)^{\epsilon\ell(w)}u_wu_{n-i}$.
  \details{
    In the details of the proof of Proposition~\ref{prop:graded-nilhecke-strong}, we show that
    \[
      u_w u_j = (-1)^{\epsilon\ell(w)}u_{w(j)}u_w\quad\text{ when }w(j+1)=w(j)+1.
    \]
    Now, setting $j=n-i$, it is easy to check that for the choice of $w$ such that $w_0=ws_{n-i}$, we have indeed that $w(n-i)=i$ and $w(n-i+1)=i+1=w(n-i)+1$.
  }
  Hence, for all $w\in S_n$,
  \[
    \tr_n(u_iu_w)=(-1)^{\bar{u}_i\bar{u}_w}\tr_n(u_wu_{n-i}),
  \]
  and the statement about the Nakayama automorphism follows.
\end{proof}

\begin{prop}\label{prop:graded-nilhecke-strong}
  The tower $\Ner_n$ is strong with twist $(\chi,d,\epsilon)$, conjugation $\Psi$ and shift $(\delta,\sigma)$, where $\chi'(n,m)=nm$ for $n,m \in \N$, $\chi''=0$, $\Psi$ is given by the Nakayama automorphism (see Proposition~\ref{prop:Frobenius-twisted}), $\delta = (\delta_n)_{n \in \N} = \left( d {n \choose 2} \right)_{n \in \N}$ and $\sigma = (\sigma_n)_{n \in \N} = \left( \epsilon {n \choose 2} \right)_{n \in \N}$.
\end{prop}

\begin{proof}
  By Lemma~\ref{lem:grad-nilcox-frob} and Proposition~\ref{prop:Frobenius-twisted}, we see that induction is conjugate shifted right adjoint to restriction with shift $(\delta,\sigma)$.  We have
  \begin{equation} \label{eq:sigma-delta-biadd}
    \kappa_\delta(n,m) = \delta_{n+m}-\delta_n-\delta_m = d {n+m \choose 2} - d {n \choose 2} - d {m \choose 2} = dnm.
  \end{equation}
  Similarly, $\kappa_\sigma(n,m)=\epsilon nm$.  Hence condition~\ref{item:S1} is satisfied by choosing $\kappa(n,m)=nm$.

  The verification of~\ref{item:S2} is almost identical to its verification in the proof of Proposition~\ref{prop:wreath-tower-strong}.  The only modification is that one takes $B=\F$ and replaces the superbimodule isomorphism~\eqref{eq:wreath-S2-isom} by a shifted version:
  \begin{equation}\label{eq:spin-nilcox-bimod} \ts
    _{(k,\ell)}(\Ner_K)_{(n,m)} \cong \bigoplus_{r=n-\ell}^{\min\{n,k\}} M_r \{d(n-r)(k-r), \epsilon (n-r)(k-r)\}.
  \end{equation}
  Similarly, the map~\eqref{eq:wreath-strong-tower-bimodule-map} is replaced by the degree $(d(n-r)(k-r),\epsilon(n-r)(k-r))$ superbimodule map
  \[
    1_{\Ner_k \otimes \Ner_\ell} \otimes 1_{\Ner_n \otimes \Ner_m} \mapsto u_{w_r}.
  \]
  The proof that this is a superbimodule isomorphism uses, instead of~\eqref{eq:wr-commutation}, the relation
  \begin{equation} \label{eq:uwr-commutation}
    u_{w_r} (a_1 \otimes a_2 \otimes a_3 \otimes a_4) =(-1)^{\epsilon (n-r)(k-r)(\bar a_1 + \bar a_2 + \bar a_3 + \bar a_4) + \bar a_2 \bar a_3} (a_1 \otimes a_3 \otimes a_2 \otimes a_4) u_{w_r}
  \end{equation}
  for all $a_1 \in \Ner_r$, $a_2 \in \Ner_{n-r}$, $a_3 \in \Ner_{k-r}$, $a_4 \in \Ner_{\ell+r-n}$.
  \details{
    Note that $\ell(w_r) = (n-r)(k-r)$.  Relation~\eqref{eq:uwr-commutation} follows from the fact that
    \begin{gather*}
      u_{w_r} u_i = (-1)^{\epsilon\ell(w_r)}u_i u_{w_r} \quad \text{if } 1 \le i < r \text{ or } n+k-r < i < K, \\
      u_{w_r} u_i =(-1)^{\epsilon\ell(w_r)} u_{i-r+k} u_{w_r} \quad \text{if } r < i < n, \\
      u_{w_r} u_i =(-1)^{\epsilon\ell(w_r)} u_{i-n+r} u_{w_r} \quad \text{if } n < i < n+k-r.
    \end{gather*}
    To see this, it suffices to show that $u_w u_i = (-1)^{\epsilon\ell(w)}u_{w(i)}u_w$ if
    \[ \tag{$*$}
      w(i+1)=w(i)+1.
    \]
    When $\epsilon=0$, this is just the same as in \cite[proof of Prop.~4.3]{SY13}, hence we may assume that $\epsilon=1$.

    We prove the result by induction on $\ell(w)$.  If $\ell(w)=$1, the condition can only be satisfied if $w=s_j$ with $|i-j|>1$.  In this case, we have
    \[
      u_j u_i = -u_i u_j = (-1)^1 u_{s_j(i)} u_j.
    \]
    Next, suppose $\ell(w)=2$ and write $w=s_ks_j$.  If $s_j$ commutes with $s_i$, then so must $s_k$ in order to have $w(i+1) = w(i)+1$.  Hence, we are done by inductive hypothesis.  Otherwise, we have either $w=s_is_{i-1}$ or $w=s_is_{i+1}$.  If $w=s_i s_{i-1}$, then
    \[
      u_w u_i = (\pm u_iu_{i-1}) u_i = \pm u_{i-1} u_i u_{i-1} = u_{i-1}(\pm u_iu_{i-1}) = u_{w(i)}u_w = (-1)^2 u_{w(i)} u_w.
    \]
    The case $w=s_is_{i+1}$ is analogous.

    Finally, suppose $\ell(w) \geq 3$, and write $w = w's_j$ with $\ell(w') < \ell(w)$.  Notice that $j \neq i$.  There are two cases.

    \textbf{Case 1}: If $j\neq i-1,i+1$, then $(w's_j)(i+1)=(w's_j)(i)+1$ implies that $w'(i+1)=w'(i)+1$. We then get
    \begin{align*}
      u_wu_i&=\pm u_{w'}u_ju_i \\
      &=\pm u_{w'}(-1)u_iu_j \\
      &=\pm (-1)^{\ell(w')}(-1)u_{w'(i)}u_{w'}u_j \qquad \text{(by the inductive hypothesis)}\\
      &= (-1)^{\ell(w')+1}u_{w'(i)}(\pm u_{w'}u_j) \\
      &= (-1)^{\ell(w)}u_{w(i)}u_w
    \end{align*}

    \textbf{Case 2}: We will assume $j=i-1$.  The case $j=i+1$ is entirely analogous.  We then have $w=w' s_{i-1}$.  Hence, $w's_{i-1}(i+1)=w's_{i-1}(i)+1$, so that $w'(i+1)=w'(i-1)+1$.  Thus, either $w'(i)>w'(i+1)$ or $w'(i)<w'(i-1)$.  In both cases, $w''=w's_i$ has one less inversion than $w'$, hence $\ell(w')=\ell(w'')+1$ and $w'=w''s_i$. Also
    \[ \tag{$**$}
      w''(i)=w's_i(i)=w'(i+1)=w'(i-1)+1=w's_i(i-1)+1=w''(i-1)+1.
    \]
    Hence $w''$ satisfies ($*$) with $i-1$ in place of $i$, and $\ell(w'')=\ell(w)-2$. Also notice that
    \[
      w''(i-1) = w's_i(i-1) = w'(i-1) = ws_{i-1}(i-1) = w(i).
    \]
    We then have
    \begin{align*}
      u_wu_i &= \alpha(w',s_{i-1})u_{w'}u_{i-1}u_i \\
      &= \alpha(w',s_{i-1})\alpha(w'',s_i)u_{w''}u_{i}u_{i-1}u_i \\
      &= \alpha(w',s_{i-1})\alpha(w'',s_i)u_{w''}u_{i-1}u_{i}u_{i-1} \\
      &= \alpha(w',s_{i-1})\alpha(w'',s_i)(-1)^{\ell(w'')} u_{w''(i-1)}u_{w''}u_iu_{i-1} \quad \text{ (by the inductive hypothesis and ($**$)) } \\
      &= (-1)^{\ell(w'')+2}u_{w(i)}\alpha(w',s_{i-1})(\alpha(w'',s_i)u_{w''}u_i)u_{i-1} \\
      &= (-1)^{\ell(w)}u_{w(i)}\alpha(w',s_{i-1})u_{w'}u_{i-1} \\
      &= (-1)^{\ell(w)}u_{w(i)}u_w.
    \end{align*}
  }
\end{proof}

\begin{lem} \label{lem:spin-nilhecke-coprod-twisting}
  We have an isomorphism of functors $\Psi^{\otimes 2} \Delta \Psi^{-1} \cong S_{12} \Delta$ on $\Ner\grm$, hence also on $\Ner\grp$.
\end{lem}

\begin{proof}
  The proof is the same as \cite[Lemma~4.4]{SY13}.
\end{proof}

%%%%%%%%%%%%%%%%%%%%%%%%%%%%%%%%%%%%%%
%\subsection{The twisted Hopf algebras}
%%%%%%%%%%%%%%%%%%%%%%%%%%%%%%%%%%%%%%

Recall that we are working over the ring $\Z_{q,\pi}$.  For a positive integer $n$, we define the $q^d\pi^\epsilon$-integer $[n] := 1 + q^d\pi^\epsilon + \dotsb + (q^d\pi^\epsilon)^{n-1}$.  We also define the $q^d\pi^\epsilon$-factorial $[n]!:=\ts\prod_{i=1}^n [i]$ and the $q^d\pi^\epsilon$-binomial $\qbin{n}{k}:=\frac{[n]!}{[k]![n-k]!}$.  By convention, we set $[0]=0$ and $[0]!=1$.  Notice that, for all $n,k\in\N$, we have $\qbin{n}{k}\in\Z_{q,\pi}$. It is straightforward to verify that $\grdim \Ner_n=[n]!$.

Since $L_n$ and $L_m$ are one dimensional over $\F$, we have
\[
  \qdim \left( \Ner_{n+m} \otimes_{\Ner_n\otimes \Ner_m} (L_n \boxtimes L_m) \right) = \qbin{n+m}{n}.
\]
Since $\Ner_{n+m}$ has a unique one-dimensional simple supermodule $L_{n+m}$, this implies that, in the Grothendieck group, we have
\[
  \left[ \Ind^{\Ner_{n+m}}_{\Ner_n\otimes \Ner_m}L_n\boxtimes L_m \right] = \qbin{n+m}{n}[L_{n+m}].
\]
We also have the obvious isomorphism of projective graded $\Ner_{n+m}$-supermodules
\[
  P_{n+m} = \Ner_{n+m}\cong \Ner_{n+m} \otimes_{\Ner_n \otimes \Ner_m} (\Ner_n \otimes \Ner_m) = \Ind^{\Ner_{n+m}}_{\Ner_n\otimes \Ner_m} (P_n \boxtimes P_m).
\]
We thus have isomorphisms of algebras
\begin{gather*}
  \cK(\Ner) \cong \Z_{q,\pi}[x],\quad [P_n] \mapsto x^n, \\
  \cG(\Ner) \cong \Z_{q,\pi}[y_1,y_2,\dotsc]/\left( y_ny_m-\ts\qbin{n+m}{n}y_{n+m} \right),\quad [L_n] \mapsto y_n.
\end{gather*}
For the restriction, note that $\Res^{\Ner_{n+m}}_{\Ner_n\otimes \Ner_m}L_{n+m}$ is a one-dimensional supermodule over $\Ner_n\otimes \Ner_m$.  Thus, it is has to be isomorphic to $L_n\boxtimes L_m$, and in the Grothendieck group we have
\begin{equation}
  \left[ \Res^{\Ner_{n+m}}_{\Ner_n\otimes \Ner_m}L_{n+m} \right] = [L_n] \otimes [L_m] \in \cG(A)\otimes\cG(A).
\end{equation}
Note that $\Ner_{n+m}$ is a free $(\Ner_n\otimes \Ner_m)$-supermodule (either on the left or on the right) of rank ${n+m \choose n}$. More precisely,  we see that
\[ \ts
  \Res^{\Ner_{n+m}}_{\Ner_n\otimes \Ner_m}P_{n+m} =\, _{\Ner_n\otimes \Ner_m}\Ner_{n+m}\otimes_{\Ner_{n+m}} P_{n+m}
  \cong \bigoplus_{w \in (S_n\times S_m) \backslash S_{n+m}} (\Ner_n\otimes \Ner_m)u_w.
\]
By considering the degrees of the basis elements $u_w$ of $\Ner_{n+m}$ as an $(\Ner_n\otimes \Ner_m)$-supermodule, we see that, in the Grothendieck group,
\begin{equation} \label{coprod}
  \left[ \Res^{\Ner_{n+m}}_{\Ner_n\otimes \Ner_m}P_{n+m} \right] = \qbin{n+m}{n}([P_n]\otimes [P_m])\in \cK(\Ner) \otimes \cK(\Ner).
\end{equation}
We can summarize this discussion in the following proposition.

\begin{prop} \label{prop:isos-alg-coalg}
  There are isomorphisms of twisted bialgebras
  \[
    \cK(\Ner) \cong \Z_{q,\pi}[x] \quad \text{and} \quad \cG(\Ner) \cong \Z_{q,\pi}[y_1,y_2,\dotsc]/(y_ny_m-\ts\qbin{n+m}{n}y_{n+m}),
  \]
  where the coproducts are given by
  \[ \ts
    \Delta(x^n)=\sum_{k=0}^n\qbin{n}{k} x^k\otimes x^{n-k} \quad \text{and} \quad \Delta(y_n)=\sum_{k=0}^n y_k\otimes y_{n-k}.
  \]
\end{prop}

\begin{cor}
  The tower $\Ner$ is compatible.
\end{cor}

\begin{proof}
  By Proposition \ref{prop:isos-alg-coalg}, $\cK(\Ner)$ is cocommutative, hence, by Lemma \ref{lem:spin-nilhecke-coprod-twisting} and Proposition \ref{prop:dualizing-twisting-isom}, $\Ner$ is dualizing, with a $(q^d\pi^\epsilon,0,\kappa)$-twisted Hopf pairing, where $\kappa(n,m)=nm$ by \eqref{eq:sigma-delta-biadd}.  In the proof of Proposition~\ref{prop:graded-nilhecke-strong}, we saw that $\cG(\Ner)$ is a $(q^d \pi^\epsilon,\kappa,0)$-Hopf algebra; but since it is cocommutative, it is also a $(q^d \pi^\epsilon,0,\kappa)$-Hopf algebra (see~\cite[Rem.~2.4]{RS14}).  Hence~\eqref{eq:compatibility} is satisfied, and so $\Ner$ is a compatible tower.
\end{proof}

The pairing~\eqref{eq:tower-pairing} satisfies $\langle x^m,y_n \rangle = \langle [P_m] , [L_n] \rangle =\delta_{mn}$.  Therefore, $x^*\left(y_m\right) = y_{m-1}$ and so the twisted Heisenberg double $\fh = \fh(\Ner)$ is the $\Z_{q,\pi}$-subalgebra of $\End \Z_{q,\pi}[y_1,y_2,\dotsc]/(y_ny_m-\ts\qbin{n+m}{n}y_{n+m})$ generated by $y_1,y_2,\dotsc$ and $x^*$.

Notice that
\[
  \cG_\pj(A) \cong \Z_{q,\pi}[y_1],
\]
and the Cartan map $\cK(A) \to \cG(A)$ of Definition~\ref{def:G-proj} corresponds to the antilinear map
\[
  \Z_{q,\pi}[x] \hookrightarrow \Z_{q,\pi}[y_1,y_2,\dotsc]/\left(y_ny_m-\ts\qbin{n+m}{n}y_{n+m}\right),\quad x\mapsto y_1.
\]
Since
\[
  \langle x^m,y_1^n \rangle = \langle [P_m] , [n]![L_n] \rangle =\delta_{mn}[n]!,
\]
we have $x^*\left(y_1^n\right) =[n] y_1^{n-1}$.  For simplicity, define $y = y_1$.  Then $x^*=\partial_y$ corresponds to the quantum partial derivation by $y$.  Thus, we have an isomorphism of algebras
\[
  \fh_\pj = \fh_\pj(\Ner) \cong \Z_{q,\pi} \langle y, \partial\ |\ \partial y = q^d\pi^\epsilon y \partial + 1\rangle.
\]
The Fock space $\cF_\pj$ is the representation of ${\mathfrak{h}_\pj}$ given by its natural action on $\Z_{q,\pi}[y]$.

\begin{rem}
  The algebras $\fh$ and $\fh_\pj$ are closely related to the quantum Weyl algebra.  In particular, when $\epsilon=0$, the numbers $[n]!$ are not zero divisors and we have an isomophism
  \[
    \cG(A) \cong \Z_{q,\pi}[x^n/[n]!],\quad y_n \mapsto x^n/[n]!.
  \]
  In this case, $\fh_\pj$ in an integral form of the rank one quantum Weyl algebra with parameter $q^d$.
\end{rem}

Since the tower $\Ner$ is strong and compatible, Theorem~\ref{theo:categorification} provides a categorification of the polynomial representation of the quantum Weyl algebra when $\epsilon=0$ and $d=1$.  More precisely, let $\F_i$ denote the trivial $N_i^{1,0}$-supermodule for $i=0,1$.  Then~\eqref{cat-eq:cross} becomes
\begin{align*}
  &\Res_{\F_1} \circ \Ind_{\F_1} = \bigoplus_n \left( \Ind_{A_1 \otimes A_{n-1}}^{A_{n}} \circ \Res_{\F_0 \boxtimes \F_1}^\dagger (\F_1 \otimes -) \right) \oplus \left( \Ind_{A_0 \otimes A_n}^{A_n} \circ \Res_{\F_1 \boxtimes \F_0}^\dagger (\F_1 \otimes -) \right) \\
  &\cong \bigoplus_n \Big(\Ind_{A_1 \otimes A_{n-1}}^{A_n} \left( \HOM_{A_0 \otimes A_1} \left( \F_0 \boxtimes \F_1, S_{23} \left( \F_1 \boxtimes \F_0 \boxtimes \Res^{A_n}_{A_{n-1} \otimes A_1}(-)\right)\right)\right)
  \{\kappa(0,1)+\chi''(1-0,1),0\} \\
  &\pushright{\oplus \Ind_{A_0 \otimes A_n}^{A_n} \left( \Hom_{A_1 \otimes A_0} \left( \F_1 \boxtimes \F_0, S_{23} \left(\F_0 \boxtimes \F_1 \boxtimes \Res^{A_n}_{A_n \otimes A_0}(-) \right) \right) \right)
  \{\kappa(1,0)+\chi''(1-1,0),0\} \Big)} \\
  &\cong \left(\Ind_{\F_1} \circ \Res_{\F_1}\right)\{1,0\} \oplus \id,
\end{align*}
which is the categorification of the defining relation $\partial x = qx \partial + 1$ of the quantum Weyl algebra.  Note that $\Res_{\F_1} = \bigoplus_n \Res^{A_{n+1}}_{A_n}$ and $\Ind_{\F_1} = \bigoplus_n \Ind^{A_{n+1}}_{A_n}$.

\begin{rem} \label{rem:Khovanov-quantum-Weyl}
  A categorification of the quantum Weyl algebra along the lines described here was outlined in~\cite[\S4.1]{Kho01}.  However, while the main functor isomorphism~\cite[Prop.~16]{Kho01} is correct, some of the other statements of~\cite[\S4.1]{Kho01} seem to be false.  In particular, the shifted restriction functor defined there does not yield a coalgebra structure.  This is a result of the fact that the shift on the restriction from $\Ner_n$ to $\Ner_k \otimes \Ner_{n-k}$ is taken to be $n-k$.  In order to obtain a coalgebra structure, one can either shift by zero (as presented above) or by $k(n-k)$ (see~\cite[Prop.~5.1]{RS14}).  In addition, even with the correct shift, one does not obtain a Hopf algebra structure on the Grothendieck groups.  Rather, one gets a twisted Hopf algebra.
\end{rem}

%%%%%%%%%%%%%%%%%%%%%%%%%%%%%%%%%%%%%%%%%%%%%%%%%%%%%%%%%%%%%%%%%%%
% References
%%%%%%%%%%%%%%%%%%%%%%%%%%%%%%%%%%%%%%%%%%%%%%%%%%%%%%%%%%%%%%%%%%%

% For BibLaTeX
%\printbibliography

\bibliographystyle{alpha}
\bibliography{RossoSavage-biblist}

\def\cprime{$'$} \newcommand{\arxiv}[1]{\href{http://arxiv.org/abs/#1}{\tt
  arXiv:\nolinkurl{#1}}}
\begin{thebibliography}{Lam99}

\bibitem[BL09]{BL09}
N.~Bergeron and H.~Li.
\newblock Algebraic structures on {G}rothendieck groups of a tower of algebras.
\newblock {\em J. Algebra}, 321(8):2068--2084, 2009.

\bibitem[CL12]{CL12}
S.~Cautis and A.~Licata.
\newblock Heisenberg categorification and {H}ilbert schemes.
\newblock {\em Duke Math. J.}, 161(13):2469--2547, 2012.

\bibitem[Kho]{Kho10}
M.~Khovanov.
\newblock Heisenberg algebra and a graphical calculus.
\newblock \arxiv{1009.3295} [math.RT].

\bibitem[Kho01]{Kho01}
M.~Khovanov.
\newblock Nil{C}oxeter algebras categorify the {W}eyl algebra.
\newblock {\em Comm. Algebra}, 29(11):5033--5052, 2001.

\bibitem[Kle05]{Kle05}
A.~Kleshchev.
\newblock {\em Linear and projective representations of symmetric groups},
  volume 163 of {\em Cambridge Tracts in Mathematics}.
\newblock Cambridge University Press, Cambridge, 2005.

\bibitem[Lam99]{Lam99}
T.~Y. Lam.
\newblock {\em Lectures on modules and rings}, volume 189 of {\em Graduate
  Texts in Mathematics}.
\newblock Springer-Verlag, New York, 1999.

\bibitem[LS12]{LS12}
A.~Licata and A.~Savage.
\newblock A survey of {H}eisenberg categorification via graphical calculus.
\newblock {\em Bull. Inst. Math. Acad. Sin. (N.S.)}, 7(2):291--321, 2012.

\bibitem[LS13]{LS13}
A.~Licata and A.~Savage.
\newblock Hecke algebras, finite general linear groups, and {H}eisenberg
  categorification.
\newblock {\em Quantum Topol.}, 4(2):125--185, 2013.

\bibitem[RS]{RS14}
D.~Rosso and A.~Savage.
\newblock Twisted {H}eisenberg doubles.
\newblock {\em Comm. Math. Phys.}
\newblock To appear, availabe at
  \url{http://dx.doi.org/10.1007/s00220-015-2330-z}.

\bibitem[SY15]{SY13}
Alistair Savage and Oded Yacobi.
\newblock Categorification and {H}eisenberg doubles arising from towers of
  algebras.
\newblock {\em J. Combin. Theory Ser. A}, 129:19--56, 2015.

\end{thebibliography}

\end{document}